%% file: LimitDir.tex
\documentclass[a4paper]{amsart}
\usepackage{todonotes}
\usepackage{a4wide}
\usepackage{amsmath}
\usepackage{amsthm}
\usepackage{amsfonts}
\usepackage{amssymb}
\usepackage{mathrsfs}
\usepackage{tikz}
\usepackage{url}
\usepackage[bookmarks=false,pdfborder={0 0 0.05}]{hyperref}
\hypersetup{colorlinks=true, citecolor=darkblue, linkcolor=darkblue}

\usepackage{graphicx}
\usepackage[labelfont=up]{caption}
\usepackage{subcaption} 
\usetikzlibrary{arrows,positioning,decorations.markings,calc,shapes,intersections,spy}

\theoremstyle{plain}
\newtheorem{theorem}{Theorem}[section]
\newtheorem{lemma}[theorem]{Lemma}
\newtheorem{corollary}[theorem]{Corollary}
\newtheorem{proposition}[theorem]{Proposition}
\newtheorem{conjecture}[theorem]{Conjecture}
\newtheorem{problem}[theorem]{Problem}

\theoremstyle{definition}
\newtheorem{definition}[theorem]{Definition}
\newtheorem{remark}[theorem]{Remark}
\newtheorem{example}[theorem]{Example}

\definecolor{darkblue}{rgb}{0,0,0.7} 
\newcommand{\darkblue}{\color{darkblue}} 
\newcommand{\Dfn}[1]{\emph{\darkblue #1}} 

\def \B {\mathcal B}
\def \L {\mathscr L}
\def \H {\mathscr H}
\def \U {\mathscr U}
\def \T {\mathscr T}
\def \e {\mathbf e}
\def \b {\mathbf b}

\def \z {\mathbf z}
\def \x {\mathbf x}
\def \y {\mathbf y}
\def \w {\mathbf w}
\def \proj #1{\widehat{#1}}
\def \closure #1{\overline{#1}}
\def \conjugate #1{\overline{#1}}
\def \hyper {\mathrm{hyp}}
\def \para {\mathrm{par}}
\def \spec {\infty}

\DeclareMathOperator{\cone}{\mathsf{cone}}
\DeclareMathOperator{\sign}{sgn}
\DeclareMathOperator{\affine}{\mathsf{aff}}
\DeclareMathOperator{\convex}{\mathsf{conv}}
\DeclareMathOperator{\support}{\mathsf{supp}}
\DeclareMathOperator{\inversion}{\mathsf{inv}}

\title[Limit Directions for Lorentzian Coxeter Systems]{Limit Directions for Lorentzian Coxeter Systems}

\author[H.~Chen]{Hao Chen$^{ \dagger}$}
\address[H.~Chen]{Freie Universit\"at Berlin, Institut f\"ur Mathematik, Arnimallee 2, 14195 Berlin, Deutschland}
\email{haochen@math.fu-berlin.de}
\urladdr{http://page.mi.fu-berlin.de/haochen}
\thanks{$^\dagger$supported by the Deutsche Forschungsgemeinschaft within the Research Training Group ``Methods for Discrete Structures'' (GRK 1408).}

\author[J.-P.~Labb\'e]{Jean-Philippe Labb\'e$^{ \ddagger}$}
\address[J.-P. Labb\'e]{Freie Universit\"at Berlin, Institut f\"ur Mathematik, Arnimallee~2, 14195 Berlin, Deutschland}
\email{labbe@math.fu-berlin.de}
\urladdr{http://page.mi.fu-berlin.de/labbe}
\thanks{$^\ddagger$supported by a FQRNT Doctoral scholarship and SFB Transregio ``Discretization in Geometry and Dynamics'' (TRR 109).}


\begin{document}

\begin{abstract}
Every Coxeter group admits a geometric representation as a group generated by reflections in a real vector space. In the projective 
representation space, \emph{limit directions} are limits of injective sequences in the orbit of some base point. Limit roots are limit 
directions that can be obtained starting from simple roots. In this article, we study the limit directions arising from \emph{any} 
point when the representation space is a Lorentz space. In particular, we characterize the light-like limit directions using 
eigenvectors of infinite-order elements. This provides a spectral perspective on limit roots, allowing for efficient computations. 
Moreover, we describe the space-like limit directions in terms of the projective Coxeter arrangement.
\end{abstract}

\maketitle

%
%

\section{Introduction}\label{sec:Intro}

Every Coxeter system has a linear representation as a reflection group acting on a real vector space endowed with a canonical bilinear form, see~\cite{bourbaki_elements_1968} or~\cite{humphreys_reflection_1992}. The unit basis vectors of the representation space are called simple roots. The set of vectors in the orbits of simple roots under the action of the Coxeter group are called roots. The set of roots forms the root system associated to the Coxeter system. Every root corresponds to a unique orthogonal reflecting hyperplane. The set of reflecting hyperplanes is called the Coxeter arrangement.

For infinite Coxeter groups, Vinberg introduced more general representations using different bilinear forms for the representation space~\cite{vinberg_discrete_1971}. We use the notation $(W,S)_{B}$ to denote a Coxeter system~$(W,S)$ associated with a matrix $B$ that determines the bilinear form used to represent~$(W,S)$. We refer to $(W,S)_B$ as a \emph{geometric Coxeter system}. When the bilinear form has signature~$(n-1,1)$, where~$n$ is the rank of $(W,S)$, the representation space is a Lorentz space. In this case, we say that~$(W,S)_{B}$ is a \emph{Lorentzian Coxeter system}. The action of the Coxeter group on Lorentz space induces an action on hyperbolic space. Coxeter groups acting on hyperbolic space with compact or finite volume fundamental domains are well studied; see for instance~\cite{vinberg_discrete_1971} \cite[Sections~{6.8-9}]{humphreys_reflection_1992} \cite[Chapter~7]{ratcliffe_foundations_2006} and \cite[Sections~{10.3-4}]{abramenko_buildings_2008}.

Let $(W,S)_B$ be a geometric Coxeter system. The Coxeter group $W$ acts linearly on the representation space $V$, therefore one may also consider its action on the corresponding projective space $\mathbb{P}V$. A point of $\mathbb{P}V$ is a \emph{limit direction} of $(W,S)_B$ if it is the limit of an injective sequence of points in the orbit of some \emph{base point} $\proj\x_0\in\mathbb{P}V$. When the base point is a simple root, then the limit direction is called a \emph{limit root}. 

The notion of limit roots was introduced and studied in \cite{hohlweg_asymptotical_2014}. Properties of limit roots for infinite Coxeter systems were investigated in a series of papers: Limit roots lie on the isotropic cone of the bilinear form associated to the representation space~\cite[Theorem~2.7]{hohlweg_asymptotical_2014}. The convex cone spanned by limit roots is studied in~\cite[Theorem~5.4]{dyer_imaginary_2013} as the \emph{imaginary cone}. The relations between limit roots and the imaginary cone are further investigated in~\cite{dyer_imaginary2_2013}. For irreducible Coxeter systems, every limit root can be obtained from any root or from any limit root \cite[Theorem~3.1(b,c)]{dyer_imaginary2_2013}.

When the Coxeter group acts on a Lorentz space, every limit root can be obtained from any time-like vector~\cite[Theorem~3.3]{hohlweg_limit_2013}. This result is obtained by interpreting the Coxeter group as a Kleinian group acting on the hyperbolic space. Furthermore, every limit root can be obtained from weights~\cite[Theorem~3.4]{chen_lorentzian_2013}. Theorem~\ref{thm:lightlike} of the present paper implies that limit directions arising from light-like directions are limit roots.

Theorem \ref{thm:limitroots} below summarizes known results on base points whose orbits accumulate at limit roots under the action of the Coxeter group. These results motivate the investigation of limit directions of Coxeter groups arising from \emph{any} base point $\proj\x_0\in\mathbb{P}V$. In this paper, we initiate this examination when the representation space $V$ is a Lorentz space. This approach includes the action of Coxeter groups on hyperbolic spaces, roots and weights as special cases. Interestingly, limit directions reveal behaviors that did not get noticed previously.

Based on the classification of Lorentzian transformations according to their eigenvalues, our first main result introduces a spectral perspective for limit roots involving infinite-order elements of the group.

\begin{theorem}\label{thm:main1}
	Let $E_\spec$ be the set of directions of non-unimodular eigenvectors for infinite-order elements of a Lorentzian Coxeter system $(W,S)_B$. The set of limit roots $E_\Phi$ of $(W,S)_B$ is the closure of $E_\spec$, that is
	\[
		E_\Phi=\closure{E_\spec}.
	\]
\end{theorem} 

This spectral description provides a natural way to directly compute limit roots from the Coxeter group and its geometric action through eigenspaces of infinite-order elements. We have implemented a package in Sage \cite{sage} to do such computations; the package will be made available in Sage during the Sage days 57 in Cernay \cite{sagedays}. In the previous articles \cite{hohlweg_asymptotical_2014} \cite{dyer_imaginary2_2013} \cite{hohlweg_limit_2013} \cite{chen_lorentzian_2013}, figures showed partial approximations of the limit roots using roots or weights. In contrast, Figure~\ref{fig:limitroots_rk4} presents two examples from our computations that reveal the full picture.

\begin{figure}[!htbp]
	\centering
	\scalebox{0.9}{
	\begin{subfigure}[b]{.45\textwidth}
		\centering
		\input{limroot1.tex}
		\caption{Some 30080 limit roots of the geometric Coxeter system with the shown graph. They are obtained from infinite-order elements of length 3 and 4 and their conjugates with elements of length 1 to 9.}\label{fig:lm_rk4_1}
	\end{subfigure}
	\qquad
	\begin{subfigure}[b]{.45\textwidth}
		\centering
		\input{limroot2.tex}
		\caption{Some 28019 limit roots of the geometric Coxeter system with the shown graph. They are obtained from infinite-order elements of length 2, 3 and 4 and their conjugates with elements of length 1 to 5.}\label{fig:lm_rk4_2}
		\end{subfigure}
	}\caption{Limit roots of two geometric Coxeter systems of rank 4 seen in the affine space spanned by the simple roots.}\label{fig:limitroots_rk4}
\end{figure}
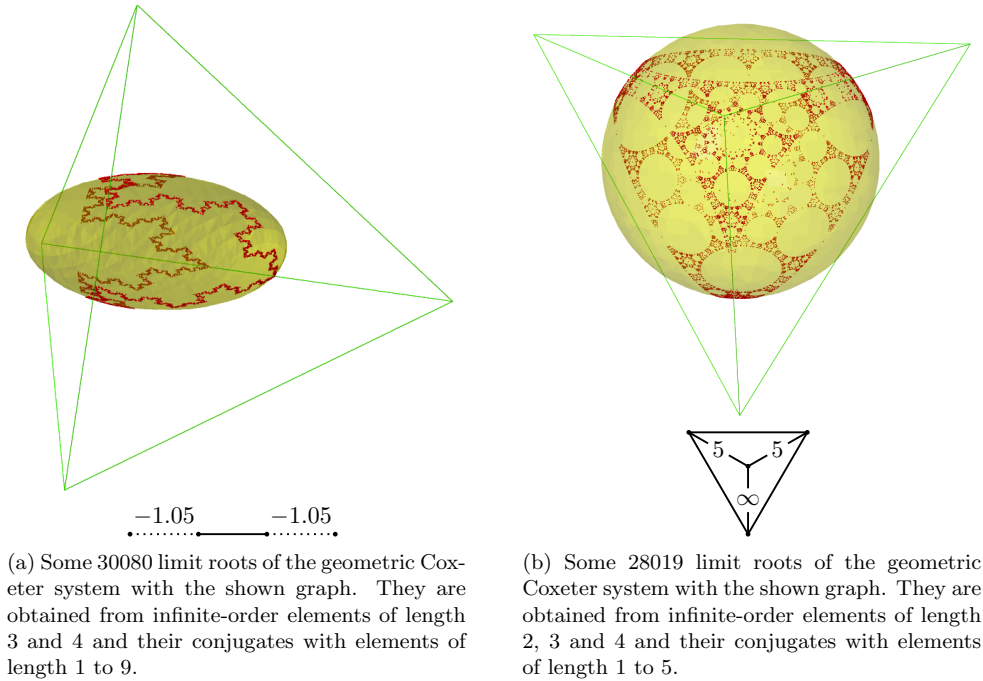

In the hope of further generalizations to non-Lorentzian Coxeter systems, our proof of Theorem~\ref{thm:main1} tries to rely minimally on the hyperbolic geometric interpretation of~\cite{hohlweg_limit_2013}.  For Lorentzian Coxeter systems, the set of limit roots equals the limit set of the Coxeter group seen as a Kleinian group.  Moreover, limit roots are indeed the only light-like limit directions, as we show in Theorem \ref{thm:lightlike}, which is a key to the proof of Theorem~\ref{thm:main1}.  However, for non-Lorentzian Coxeter systems, there may be isotropic limit directions that are not limit roots, as shown in Example~\ref{expl:nonLorentz}. 

Furthermore, while proving Theorem~\ref{thm:main1}, we also observed space-like limit directions.  Our second main result describes the set of limit directions for Lorentzian Coxeter systems in terms of the projective Coxeter arrangement, i.e.\ the infinite arrangement of reflecting hyperplanes in the projective representation space.

\begin{theorem}\label{thm:main2}
	Let $E_V$ be the set of limit directions of a Lorentzian Coxeter system $(W,S)_B$ and $\L_\hyper$ be the union of codimension-$2$ space-like intersections of the projective Coxeter arrangement associated to $(W,S)_B$.  The set of limit directions $E_V$ is ``sandwiched'' between $\L_\hyper$ and its closure, that is
	\[
		\L_\hyper\subset E_V\subseteq\closure{\L_\hyper}.
	\]
\end{theorem}

\begin{figure}[!htbp]
  \includegraphics{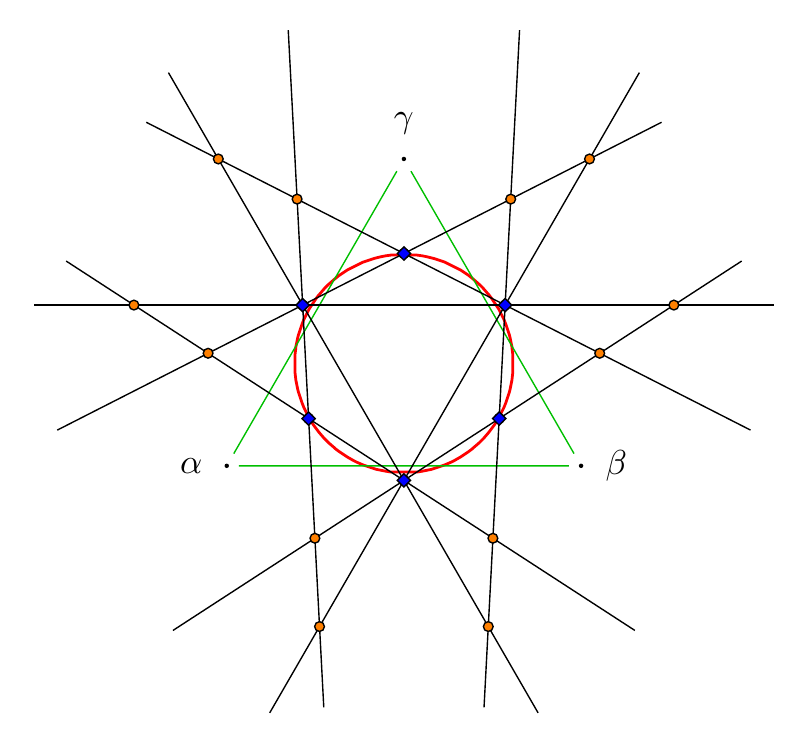}
  \caption{
    Some reflecting hyperplanes in the projective Coxeter arrangement of the universal Coxeter group of rank $3$.  The associated bilinear form has $c_{ij}=1.1$ whenever $i\ne j$. Codimension-$2$ space-like intersections are marked with dots.  By Theorem~\ref{thm:main2}, the intersections in $\L_\hyper$ are limit directions.  The six intersections marked with diamonds are weights.  
  }\label{fig:intersection}
\end{figure}

The hyperplane arrangement $\L_\hyper$ involved in Theorem~\ref{thm:main2} is infinite and not discrete.  Figure~\ref{fig:intersection} shows part of the Coxeter arrangement of a universal Coxeter group and some intersections in $\L_\hyper$. 

While Theorem~\ref{thm:main2} has a combinatorial flavour, a stronger relation holds. Let $\U_\hyper$ be the set of space-like unimodular eigendirections of infinite-order elements. Then the set of limit directions satisfies \[\U_\hyper\sqcup E_\Phi\subseteq E_V\subseteq\closure{\U_\hyper}.\] Notably, space-like weights are all limit directions, as marked by diamonds in Figure~\ref{fig:intersection} and \ref{fig:limitdirection}.  Moreover, roots may also be limit directions.  In Section~\ref{ssec:problem}, we discuss the possibility of equalities on either side, and point out certain unexplained linear dependencies among limit directions in $\U_\hyper$, see Figure~\ref{fig:limitdirection}.

After we obtained the main results, we noticed potential relations between limit roots and two conjectures: a conjecture of Lam and Pylyavskyy, stating that the \emph{limit weak order} is a lattice~\cite[Conjecture~10.3]{lam_total_2013}, and a conjecture of Dyer, stating that the \emph{extended weak order} is a complete ortholattice~\cite[Conjecture~2.5]{dyer_weak_2011}. The limit weak order extends the usual weak order to infinite reduced words, whereas the extended weak order extends the usual weak order to infinite biclosed sets.  In Section~\ref{ssec:infword}, we relate limit roots to the study of infinite reduced words and infinite biclosed sets for Lorentzian Coxeter systems. The relations between limit weak order and extended weak order in general deserve to be explored in more detail.

The present paper is organized as follows. In Section~\ref{sec:Coxeter}, we recall the geometric representations of Coxeter systems, define the notion of limit directions, and review some results on limit roots. Then we prove that limit roots are the only light-like limit directions for Lorentzian Coxeter systems, and study the relation between infinite reduced words and limit roots.  In Section~\ref{sec:Spectral}, we recall spectral properties of Lorentz transformations, and prove Theorem \ref{thm:main1}. In the last part of Section~\ref{sec:Spectral}, we give an example of non-Lorentzian Coxeter system, for which some isotropic limit directions are not limit roots. In Section~\ref{sec:Arrangement}, we define and study the projective Coxeter arrangement, and prove Theorem \ref{thm:main2}. Finally, some open problems are discussed.

\subsection*{Acknowledgement} 
The authors would like to thank Christophe Hohlweg, Ivan Izmestiev and Vivien Ripoll for helpful discussions. The second author is grateful to Christophe Hohlweg for introducing him to infinite root systems during Summer 2010 and to S\'ebastien Labb\'e and Vivien Ripoll for their help in the implementation of the Sage package.

%
%

\section{Limit roots of Lorentzian Coxeter systems}\label{sec:Coxeter}

\subsection{Lorentzian Coxeter systems}
A $n$-dimensional \Dfn{Lorentz space} $(V,\B)$ is a vector space $V$ associated with a bilinear form $\B$ of signature $(n-1,1)$. In a Lorentz space, a vector $\x$ is \Dfn{space-like} (resp.~\Dfn{time-like, light-like}) if $\B(\x,\x)$ is positive (resp.~negative, zero). The set of light-like vectors $Q=\{\x\in V\mid\B(\x,\x)=0\}$ forms a cone called the \Dfn{light cone}. The following proposition characterizes the totally-isotropic subspaces of Lorentz spaces.  

\begin{proposition}[{\cite[Theorem~2.3]{cecil_lie_2008}}]
  \label{prop:twolights}
  Let $(V,\B)$ be a Lorentz space and $\x,\y\in Q$ be two light-like vectors.  Then $\B(\x,\y)=0$ if and only if $\x=c\y$ for some $c\in\mathbb{R}$.
\end{proposition}

A linear transformation on $V$ that preserves the bilinear form $\B$ is called a \Dfn{Lorentz transformation}. The group of Lorentz transformations is called \Dfn{Lorentz group} and noted $O_{\B}(V)$.

Let $(W,S)$ be a finitely generated \Dfn{Coxeter system}, where $S$ is a finite set of generators and the \Dfn{Coxeter group}~$W$ is generated with the relations $(st)^{m_{st}}=e$, where $s,t\in S$, $m_{ss}=1$ and $m_{st}=m_{ts}\geq 2$ or $=\infty$ if $s\neq t$. The cardinality $|S|=n$ is the \Dfn{rank} of the Coxeter system~$(W,S)$. For an element $w\in W$, the \Dfn{length}~$\ell(w)$ of $w$ is the smallest natural number $k$ such that $w=s_1s_2\dots s_k$ for $s_i\in S$.  We refer the readers to \cite{bourbaki_elements_1968,humphreys_reflection_1992} for more details.  We associate a matrix $B$ to $(W,S)$ as follows: \[B_{st}= \begin{cases} -\cos(\pi/m_{st}) & \text{if}\quad m_{st}<\infty;\\ -c_{st} & \text{if}\quad m_{st}=\infty,\\ \end{cases} \] for $s,t\in S$, where $c_{st}$ are chosen arbitrarily with $c_{st}=c_{ts}\geq 1$. The Coxeter system $(W,S)$ associated with the matrix $B$ is called a \Dfn{geometric Coxeter system} and is noted by $(W,S)_B$. 

Let $V$ be a real vector space of dimension $n$, equipped with a basis $\Delta=\{\alpha_s\}_{s\in S}$. The matrix~$B$ defines a bilinear form $\B$ on $V$ by $\B(\alpha_s,\alpha_t)=\alpha_s^\intercal B\alpha_t$, for $s,t\in S$.  For a vector $\alpha\in V$ such that $\B(\alpha,\alpha)\neq 0$, we define the reflection $\sigma_\alpha$
\begin{equation}\label{eqn:Bref}
  \sigma_\alpha(\x):=\x-2\frac{\B(\x,\alpha)}{\B(\alpha,\alpha)}\alpha\quad\text{for all }\x\in V.
\end{equation}
The homomorphism $\rho:W\to\mathrm{GL}(V)$ sending $s$ to $\sigma_{\alpha_s}$ is a faithful geometric representation of the Coxeter group $W$ as a discrete group of Lorentz transformations. We refer the readers to \cite[Chapter~1]{krammer_conjugacy_2009} and \cite[Section~1]{hohlweg_asymptotical_2014} for more details. In the following, we will write $w(x)$ in place of $\rho(w)(x)$. 

A geometric Coxeter system $(W,S)_B$ is of \Dfn{finite type} if $B$ is positive-definite. In this case, $W$ is a finite group. A geometric Coxeter system~$(W,S)_B$ is of \Dfn{affine type} if the matrix~$B$ is positive semi-definite but not definite. In either case, the group $W$ can be represented as a reflection group in Euclidean space. If the matrix $B$ has signature $(n-1,1)$, the pair $(V,\B)$ is a $n$-dimensional Lorentz space, and $W$ acts on $V$ as a discrete subgroup of the Lorentz group $O_{\B}(V)$.  In this case, we say that the geometric Coxeter system $(W,S)_B$ is \Dfn{Lorentzian} and, by abuse of language, that $W$ is a \Dfn{Lorentzian Coxeter group}.  See~\cite[Remark~2.2]{chen_lorentzian_2013} for further discussions about terminologies.

In the spirit of~\cite{chen_lorentzian_2013}, we pass to the \Dfn{projective space} $\mathbb{P}V$, i.e.\ the topological space of $1$-dimensional subspaces of~$V$. For a non-zero vector $\x\in V\setminus\{0\}$, let $\proj\x\in\mathbb{P}V$ denote the line passing through $\x$ and the origin.  The group action of~$W$ on $V$ by reflection induces a \Dfn{projective action} of~$W$ on $\mathbb{P}V$ as $w\cdot\proj\x=\proj{w(\x)}$, for $w\in W$ and $\x\in V$. For a set $X\subset V$, the corresponding projective set is $\proj X:=\{\proj\x \in \mathbb{P}V\mid\x\in X\}$.  The \Dfn{projective light cone} is denoted by $\proj Q$.

Let~$h(\x)$ denote the sum of the coordinates of $\x$ in the basis $\Delta$, and call it the \Dfn{height} of the vector~$\x$. We say that $\x$ is \Dfn{future-directed} (resp. \Dfn{past-directed}) if $h(\x)$ is positive (resp.  negative). The hyperplane $\{\x\in V\mid h(\x)=1\}$ is the affine subspace $\affine(\Delta)$ spanned by the basis of $V$. It is useful to identify the projective space $\mathbb{P}V$ with the affine subspace $\affine(\Delta)$ with a \Dfn{projective hyperplane added at infinity}. For a vector $\x\in V$, if $h(\x)\ne 0$, $\proj\x$ is identified with the vector \[ \x/h(\x)\in\affine(\Delta).  \] Otherwise, if $h(\x)=0$, the direction $\proj\x$ is identified to a point on the projective hyperplane at infinity. For a basis vector $\alpha\in\Delta$, we have $\proj\alpha=\alpha$. In fact, if $h(\x)\ne 0$, $\proj\x$ is identified with the intersection of $\affine(\Delta)$ with the straight line passing through $\x$ and the origin.  The projective light cone~$\proj Q$ is projectively equivalent to a sphere, see for instance~\cite[Proposition~4.13]{dyer_imaginary2_2013}. The affine subspace~$\affine(\Delta)$ is practical for visualization and geometric intuitions.

Given a topological space $X$ and a subset $Y\subseteq X$, a point $x\in X$ is an \Dfn{accumulation point} of~$Y$ if every neighborhood of $x$ contains a point of $Y$ different from $x$. Let $G$ be a group acting on~$X$, then a point $x\in X$ is a \Dfn{limit point} of $G$ if $x$ is an accumulation point of the orbit $G(x_0)$ for some \Dfn{base point} $x_0\in X$. Alternatively, $x$ is a limit point of $G$, if there is a base point $x_0\in X$ and a sequence of elements $(g_k)_{k\in\mathbb{N}}\in G$ such that the sequence of points $(g_k(x_0))_{k\in\mathbb{N}}$ is injective and converges to $x$ as $k\to\infty$. In this case, we say that $x$ is a limit point of $G$ acting on $X$ arising from the base point $x_0$ through the sequence $(g_k)_{k\in\mathbb{N}}$.  We now define the main object of study of the present paper.

\begin{definition}
  \Dfn{Limit directions} of a geometric Coxeter system $(W,S)_B$ are limit points of $W$ acting on $\mathbb{P}V$. The set of limit directions is denoted by $E_V$. In other words,
  \begin{align*}
    E_V=\{\proj{\x}\in \mathbb{P}V \mid&\text{ there is a vector } \x_0\in V \text{ and an injective sequence } (w_i\cdot\proj\x_0)_{i\in\mathbb{N}}\\
    &\text{ in the orbit } W\cdot\proj\x_0 \text{ such that } \lim_{i\rightarrow \infty}  w_i\cdot\proj\x_0=\proj\x\}.
  \end{align*}
\end{definition}

\subsection{Limit roots}

We call the basis vectors in $\Delta$ the \Dfn{simple roots}. Let $\Phi=W(\Delta)$ be the orbit of $\Delta$ under the action of~$W$, then the vectors in $\Phi$ are called \Dfn{roots}. The pair $(\Phi,\Delta)$ is a \Dfn{based root system}. The roots~$\Phi$ are partitioned into \Dfn{positive roots} $\Phi^+=\cone(\Delta)\cap\Phi$ and \Dfn{negative roots} $\Phi^-=-\Phi^+$. The \Dfn{depth} of a positive root $\gamma\in\Phi^+$ is the smallest integer $k$ such that $\gamma=s_1s_2\dots s_{k-1}(\alpha)$, for $s_i\in S$ and $\alpha\in\Delta$. 

Let $V^*$ be the dual vector space of $V$ with dual basis $\Delta^*$. If the bilinear form $\B$ is non-singular, which is the case for Lorentz spaces, $V^*$ can be identified with $V$, and $\Delta^*=\{\omega_s\}_{s\in S}$ can be identified with a set of vectors in $V$ such that \begin{equation}\label{eqn:dualbasis} \B(\alpha_s,\omega_t)=\delta_{st}, \end{equation} where $\delta_{st}$ is the Kronecker delta function. Vectors in $\Delta^*$ are called \Dfn{fundamental weights}, and vectors in the orbit
\[
  \Omega:=W(\Delta^*)=\bigcup_{\omega\in\Delta^*}W(\omega)
\]
are called \Dfn{weights}. See \cite[Chapter~VI, Section~1.10]{bourbaki_elements_1968} and \cite[Section~1]{maxwell_sphere_1982} for more details about the notion of weights.

Limit roots were introduced in \cite{hohlweg_asymptotical_2014} to study infinite root systems arising from infinite Coxeter groups. They are the accumulation points of the projective roots $\proj\Phi$ in $\mathbb{P}V$. In other words, the set of limit roots is defined as follows
\[
  E_\Phi=\{\proj{\x}\in \mathbb{P}V \mid \text{ there is an injective sequence } (\gamma_i)_{i\in\mathbb{N}} \in \Phi \text{ such that } \lim_{i\rightarrow \infty} \proj \gamma_i=\proj{\x}\}.
\]
Dyer also studied limit roots as the boundary of the imaginary cone, see~\cite[Theorem~5.4]{dyer_imaginary_2013}. Limit roots are on the \Dfn{isotropic cone} $\{\proj\x\in\mathbb{P}V\mid\B(\x,\x)=0\}$, see~\cite[Theorem~2.7(ii)]{hohlweg_asymptotical_2014}. Limits roots are also limit directions arising from different base points. The following theorem, schematized in Figure~\ref{fig:schema_roots}, summarizes the known results. 

\begin{theorem}\label{thm:limitroots}
  Limit roots of a geometric Coxeter system $(W,S)_B$ are limit directions arising from
  \begin{enumerate}
      \renewcommand{\labelenumi}{(\roman{enumi})}
      \renewcommand{\theenumi}{\labelenumi}
    \item simple roots, see \cite[Definition~2.12]{hohlweg_asymptotical_2014},
    \item limit roots, see \cite[Theorem~3.1(b)]{dyer_imaginary2_2013},
    \item projective roots, see \cite[Theorem~3.1(c)]{dyer_imaginary2_2013}.
  \end{enumerate}
  Moreover, if $(W,S)_B$ is Lorentzian, limit roots are limit directions arising from
  \begin{enumerate}
      \renewcommand{\labelenumi}{(\roman{enumi})}
      \renewcommand{\theenumi}{\labelenumi}
      \setcounter{enumi}{3}
    \item time-like directions, see \cite[Theorem~3.3]{hohlweg_limit_2013},
    \item projective weights, see \cite[Theorem~3.4]{chen_lorentzian_2013},
    \item light-like directions, by Theorem~\ref{thm:lightlike} of this paper.
  \end{enumerate}
\end{theorem}

\begin{figure}[!htbp]
  \includegraphics{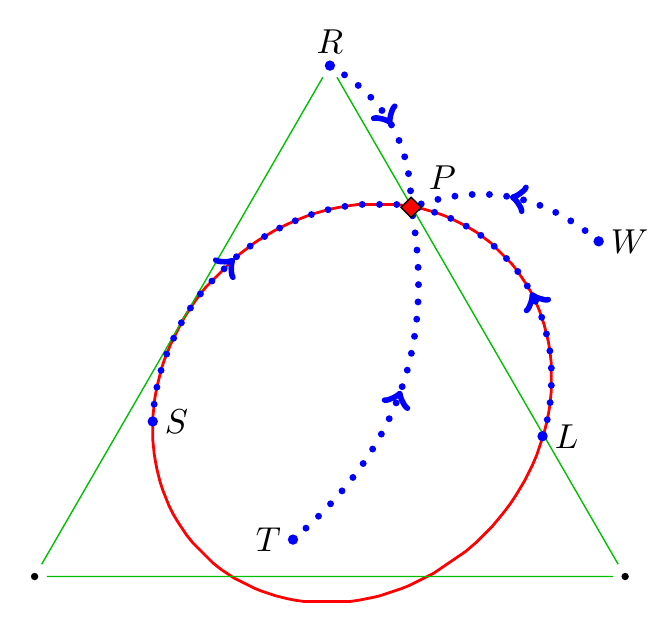}
  \caption{
    Schematic picture showing a limit root $P$ arising from different base points: time-like direction (T), light-like direction (S), projective weight (W), projective root (R) or another limit root (L). The elliptic-shape is the projective light cone, and the triangle is the convex hull of the projective simple roots.
  }\label{fig:schema_roots}
\end{figure}

\begin{remark}
  In this paper, we require the simple roots to be a basis for a based root system. However, except for the concept of weights, all the arguments and results of this paper work in the more general setting of \cite{hohlweg_asymptotical_2014,dyer_imaginary_2013,dyer_imaginary2_2013}, where the simple roots only needs to be positively independent but not necessarily linearly independent.
\end{remark}

\subsection{Light-like limit directions}\label{ssec:lightlike}

By Theorem~\ref{thm:limitroots}, the set of limit roots $E_\Phi$ is contained in~$E_V$.  The following theorem states that limit roots are the only light-like limit directions of Lorentzian Coxeter groups.

\begin{theorem}\label{thm:lightlike}
  For a Lorentzian Coxeter system $(W,S)_B$, consider a sequence $(w_k)_{k\in\mathbb{N}}\in W$ and a base point $\x_0\in V$. If $(w_k\cdot\proj\x_0)_{k\in\mathbb{N}}$ is injective and converges to a limit direction $\proj\x$ in the projective light cone $\proj Q$, then $\proj\x$ is a limit root.
\end{theorem}

\begin{proof}
  Choose a sequence of simple roots $(\alpha_k)_{k\in\mathbb{N}}$ such that $(w_k\cdot\alpha_k)_{k\in\mathbb{N}}$ is also an injective sequence of projective positive roots. Since $\Delta$ is a finite set, the sequence $(\alpha_k)_{k\in\mathbb{N}}$ visits a certain simple root, say $\alpha\in\Delta$, infinitely many times. By passing to a subsequence, we may assume that $\alpha_k=\alpha$ for all $k\in\mathbb{N}$. Since $\mathbb{P}V$ is compact, by passing again to a subsequence, we may assume that $(w_k\cdot\alpha)$ converges to a limit root $\proj\beta\in\proj Q$.  Assume that $\x_0$ is not light-like, then $h(w_k(\x_0))$ tends to infinity since
  \[
    0=\B(\proj\x,\proj\x)=\lim_{k\to\infty}\B(w_k\cdot\proj\x_0,w_k\cdot\proj\x_0)=\lim_{k\to\infty}\frac{\B(\x_0,\x_0)}{h(w_k(\x_0))^2}.
  \]
  The height $h(w_k(\alpha))$ tends to infinity (see the proof of \cite[Theorem 2.7]{hohlweg_asymptotical_2014}). Moreover, $\B(w_k(\alpha),w_k(\x_0))=\B(\alpha,\x_0)$ is constant. Therefore
  \[
    \B(\proj\x,\proj\beta)=\lim_{k\to\infty}\B(w_k\cdot\proj\x_0,w_k\cdot\alpha)=\lim_{k\to\infty}\frac{\B(\x_0,\alpha)}{h(w_k(\x_0))h(w_k(\alpha))}=0.
  \]
  Since $\proj\x$ and $\proj\beta$ are both in the projective light cone, we have $\proj\x=\proj\beta$ by Proposition \ref{prop:twolights}.  The limit direction~$\proj\x$ is therefore a limit root. Furthermore, the above argument does not depend on the choice of base point $\x_0\notin Q$. So if $\proj\y\in\proj Q$ is another limit direction arising from $\y_0\notin Q$ through the same sequence $(w_k)_{k\in\mathbb{N}}$, we have $\proj\x=\proj\y=\proj\beta\in E_\Phi$.

  If $\x_0$ is light-like, it can be decomposed as a linear combination of a time-like vector $\x'_0$ and a space-like vector $\x''_0$, see Figure~\ref{fig:lightlike} for an illustration of this case. Under the action of the sequence $(w_k)_{k\in\mathbb{N}}$, the time-like component $(w_k\cdot\proj\x'_0)_{k\in\mathbb{N}}$ converges to a limit root, see \cite[Theorem~3.3]{hohlweg_limit_2013}. Let $\proj\beta$ be this limit root.  If the space-like component $(w_k(\x''_0))_{k\in\mathbb{N}}$ does not converge to the light cone, the norm of $w_k(\x''_0)$ is bounded because the action of $W$ preserves the bilinear form. In this case, we have
  \[
    \lim_{k\to\infty}w_k\cdot\proj\x_0=\lim_{k\to\infty}w_k\cdot\proj\x'_0=\proj\beta\in E_\Phi.
  \]

  If the space-like component converges to the light cone, its direction $(w_k\cdot\proj\x''_0)_{k\in\mathbb{N}}$ also converges to the limit root $\proj\beta$. Then the sequence $(w_k\cdot\proj\x_0)_{k\in\mathbb{N}}$, being the direction of a light-like linear combination of the two components, must converge to the same limit root $\proj\beta$.
\end{proof}

As a consequence, limit directions arising from light-like directions are limit roots.

\begin{figure}[!htbp]
  \includegraphics{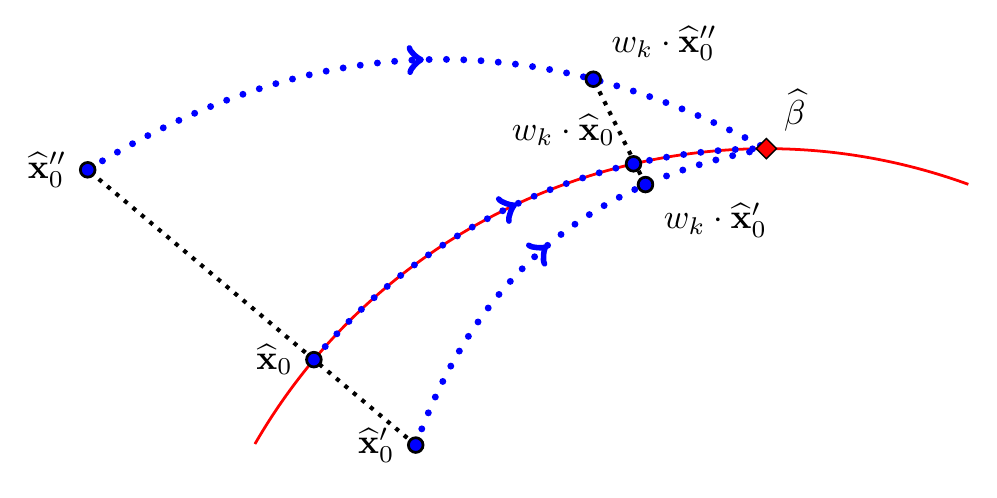}
  \caption{
    Illustration for the proof of Theorem~\ref{thm:lightlike} in the case where the base point $\x_0$ is light-like.
  }\label{fig:lightlike}
\end{figure}

\begin{corollary}[of the proof]\label{cor:samelimits}
  Limit roots arising from different base points but through the same sequence are the same.
\end{corollary}

\begin{remark}\label{rem:nonLorentz}
  If the geometric Coxeter system $(W,S)_B$ is not Lorentzian, limit roots are in general not the only limit directions on the isotropic cone. An example is given in Example~\ref{expl:nonLorentz} of Section~\ref{ssec:nonLorentz}.
\end{remark}

\subsection{Infinite reduced words}\label{ssec:infword}

An \Dfn{infinite reduced word} is an infinite sequence $\w=s_1s_2\dots$ of generators in $S$ such that every finite prefix $w_k=s_1s_2\dots s_k$ is a reduced expression, see~\cite[Section~4.2]{lam_total_2013} and \cite{lam_infinite_2013}. The \Dfn{inversion set} $\inversion(\w)\subset\Phi^+$ is the set of positive roots in the form of $w_{k-1}(\alpha_{s_k})$, $k\in\mathbb{N}$. The same definition also applies to finite words. To prove the following theorem, we need the notion of biclosed sets of roots. A subset $A$ of roots in $\Phi^+$ is \Dfn{closed} if for two roots $\alpha,\beta\in A$, any root that is a positive combination of $\alpha$ and $\beta$ is also in~$A$.  A subset~$A$ is \Dfn{biclosed} if both $A$ and its complement in $\Phi^+$ are closed. Finite biclosed sets of $\Phi^+$ are in bijections with the inversion sets of elements of $W$, see~\cite[Proposition~1.2]{pilkington_convex_2006}.  See also~\cite[Chapter~2]{labbe_polyhedral_2013} for more detail on the relation between biclosed sets and the study of limit roots.

\begin{theorem}\label{thm:inv_set}
  Let $\w$ be an infinite reduced word of a Lorentzian Coxeter system $(W,S)_B$. The projective inversion set $\proj{\inversion(\w)}$ of an infinite reduced word $\w$ has a unique limit root as its accumulation point.
\end{theorem}

\begin{proof} 
  As a set of positive root, the accumulation points of $\proj{\inversion(\w)}$ are limit roots. For the sake of contradiction, assume that $\proj{\inversion(\w)}$ accumulates at two distinct limit roots $\proj\x$ and~$\proj\y$. Because the projective light cone is strictly convex, we can find two projective roots $\proj\alpha, \proj\beta\in\proj{\inversion(\w)}$ respectively in the neighborhood of $\proj\x$ and $\proj\y$ such that the segment $[\proj\alpha,\proj\beta]$ intersect $\proj Q$.  Moreover, there is a positive integer $k>0$ such that $\alpha$ and $\beta$ are contained in $\inversion(w_k)$.  However, the reflections in $\alpha$ and $\beta$ generate an infinite dihedral group, so $\inversion(w_k)$ can not be finite and closed at the same time.
\end{proof}

Thus, we can associate a unique limit root to an infinite reduced word $\w$ and denote it by $\proj\gamma(\w)$. 

\begin{corollary}
  The limit root $\proj\gamma(\w)$ arises through the sequence of prefixes $(w_k=s_1\dots s_k)_{k\in\mathbb{N}}$ of~$\w$.
\end{corollary}

\begin{proof}
  At least one generator $s\in S$ appears in $\w$ infinitely many times, so we can take from $\inversion(\w)$ an injective subsequence $(w_k(\alpha_s))$ such that $s_{k+1}=s$.  Then $\proj\gamma(\w)$ is the limit root arising from $\alpha_s$ through the sequence $(w_k)$. By Corollary~\ref{cor:samelimits}, the same limit root arises through the same sequence from any projective root. We then conclude that $\proj\gamma(\w)$ arises through the sequence $(w_k)_{k\in\mathbb{N}}$.
\end{proof}

Consider two infinite reduced word $\w$ and $\w'$. It is easy to see that $\proj\gamma(\w)=\proj\gamma(\w')$ if $\inversion(\w)\cap\inversion(\w')$ contains infinitely many roots.  Lam and Thomas~\cite[Theorem 1(1)]{lam_infinite_2013} proved that $\w$ and $\w'$ correspond to the same set of points on the Tits boundary of the Davis complex if $\inversion(\w)$ and $\inversion(\w')$ differ by finitely many roots.  Moreover, since infinite reduced words correspond to geodesic rays in the Cayley graph of $(W,S)$, there is also a correspondence between limit roots and the group ends of $W$. In Figure~\ref{fig:inversion_ends}, the inversion set of an infinite Coxeter element and the corresponding sequence of chambers in the Tits cone are shown.  The sequence of chambers correspond to a geodesic ray in the Cayley graph of $(W,S)$.

\begin{figure}
  \includegraphics{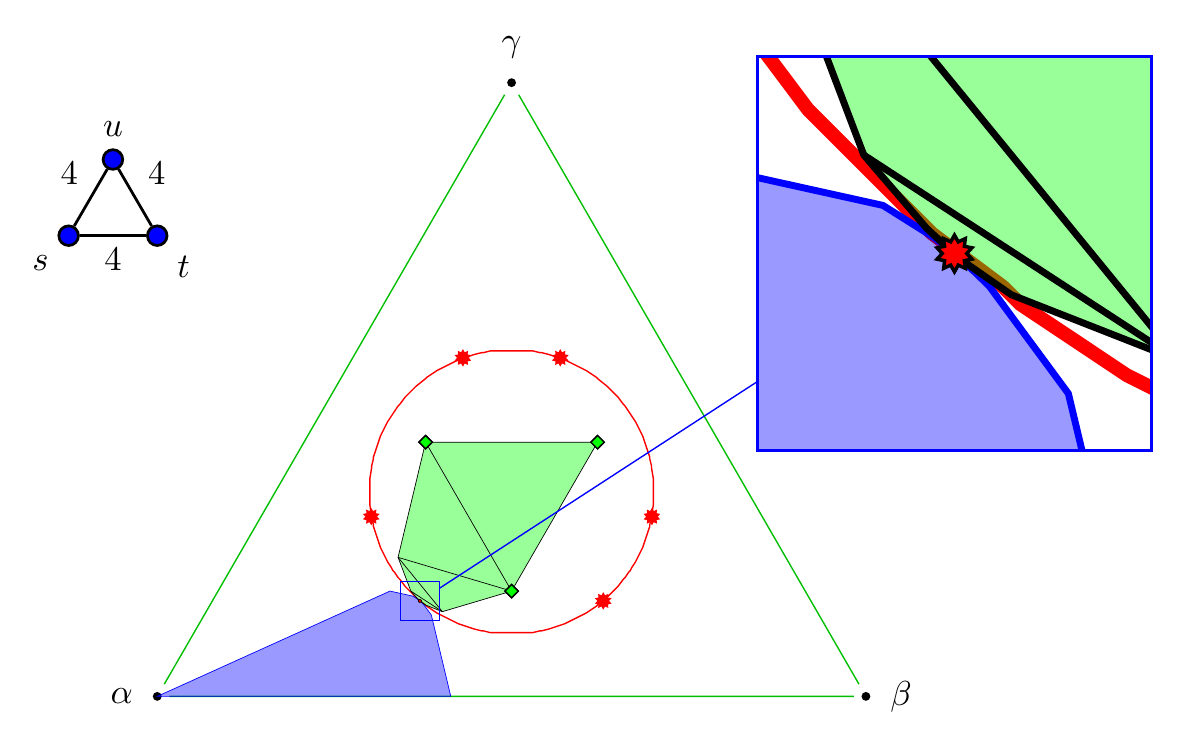}
  \caption{
    The inversion set of the infinite reduced word $(stu)^{\infty}$ represented by the convex hull of its roots, and the corresponding sequence of chambers inside the Tits cone. They are disjoint and share a unique point on their boundary: the limit root given by the dominant eigendirection of $stu$.
  }\label{fig:inversion_ends}
\end{figure}

It would be interesting to find an equivalence relation on infinite reduced words such that two words $\w$ and $\w'$ are equivalent if and only if $\proj\gamma(\w)=\proj\gamma(\w')$. None of the results above satisfy this requirement. Consider the infinite reduced words $(st)^\infty$ and $(ts)^\infty$ of an affine infinite dihedral group. Their inversion sets are disjoint, but they correspond to a same limit root.  It would make more sense if $(st)^\infty\sim(ts)^\infty$ whenever $c_{st}=1$.  A way to encode geometric information (the bilinear form $\B$) into the equivalence relation would be very helpful. Besides, it may be possible to define limit roots as the completion of limit roots obtained from infinite-order elements using a metric on sequences of infinite-order elements which respects the geometry.

On the one hand, Lam and Pylyavskyy conjectured that the \emph{limit weak order}, i.e. the finite and infinite inversion sets ordered by inclusion, for the affine Coxeter group $\tilde A_n$ forms a lattice, see~\cite[Conjecture~10.3]{lam_total_2013}. On the other hand, Dyer conjectured that the \emph{extended weak order}, i.e. the biclosed sets ordered by inclusion, forms a complete ortholattice, see~\cite[Conjecture~2.5]{dyer_weak_2011}. In view of Theorem~\ref{thm:inv_set}, it seems reasonable to use the notion of limit roots to unify both conjectures. Namely, one verifies that an infinite inversion set $\inversion(\w)$ is biclosed in the affine and Lorentzian case, otherwise it would contradict the biclosedness of the inversion set of a certain finite prefix of $\w$. The difference between the two conjectures lies in the fact that there are many biclosed sets that are neither finite nor cofinite, yet are not infinite inversion sets, see~\cite[Figure~2.11]{labbe_polyhedral_2013} for an example.  The relations between the two conjectures should be made clear and deserve better attention.

%
%

\section{Spectra of elements of Lorentzian Coxeter groups}\label{sec:Spectral}

\subsection{Spectra of Lorentz transformations}

In a Lorentz space $(V,\B)$, a subspace $U$ of $V$ is \Dfn{space-like} if its non-zero vectors are all space-like, \Dfn{light-like} if $U$ contains some non-zero light-like vector but no time-like vector, or \Dfn{time-like} if $U$ contains some time-like vector. Two vectors $\x,\y\in V$ are said to be orthogonal if $\B(\x,\y)=0$. For a vector $\x\in V$, we define its orthogonal hyperplane
\[
  H_\x=\{\y\in V\mid \B(\x,\y)=0\}.
\]
We see that $H_\x$ is space-like (resp.~light-like, time-like) if $\x$ is time-like (resp.~light-like, space-like). For a subspace $U$ of $V$, its orthogonal companion is defined as
\[
  U^\perp=\{\y\in V\mid\B(\y,\x)=0\text{ for all }\x\in U\}.
\]
If $U$ is lightlike, $U+U^\perp$ is not the whole space $V$.

To study the eigenvalues and eigenvectors of Lorentz transformations, it is useful to work in the complexification $V_\mathbb{C}=V\oplus iV$. A vector $\z\in V_\mathbb{C}$ can be written as $\z=\x+i\y$ for $\x,\y\in V$. We call~$\x$ (resp.~$\y$) the \Dfn{real part} (resp.~\Dfn{imaginary part}) of $\z$, and the vector $\conjugate\z=\x-i\y$ represents its \Dfn{conjugate vector}. If $\y\ne 0$, we say that $\z$ is \Dfn{space-like} (resp.~\Dfn{light-like}, \Dfn{time-like}) if the subspace spanned by $\x$ and $\y$ is space-like (resp.~light-like, time-like). The bilinear form $\B$ on $V$ is viewed as the restriction of a sesquilinear form on $V_\mathbb{C}$ defined by requiring in addition that 
\[
  \B(\lambda\z_1,\mu\z_2)=\lambda\conjugate\mu\B(\z_1,\z_2)
\]
for $\z_1,\z_2\in V_\mathbb{C}$ and $\lambda,\mu \in \mathbb{C}$. Then the action of the Lorentz group $O_{\B}(V)$ naturally extends to~$V_\mathbb{C}$.  Again, two vectors $\z_1,\z_2\in V_\mathbb{C}$ are \Dfn{orthogonal} if $\B(\z_1,\z_2)=0$. 

\begin{remark}
  The bilinear form $\B$ associated to $V$ can also be viewed as the restriction of a bilinear form on $V_\mathbb{C}$, as in \cite[Chapter III]{riesz_clifford_1958}. This is algebraically more natural, while a sesquilinear form is geometrically more natural.
\end{remark}

A non-zero vector $\z\in V_\mathbb{C}$ is an eigenvector of a Lorentz transformation $\phi\in O_{\B}(V)$ if $\phi(\z)=\lambda\z$ for some $\lambda\in\mathbb{C}$. An eigenvalue $\lambda$ is \Dfn{unimodular} if $|\lambda|=1$, in which case a $\lambda$-eigenvector $\z$ is also said to be unimodular. If $\x\in V$ is an eigenvector of $\phi$, we call $\proj\x\in\mathbb{P}V$ an \Dfn{eigendirection} of $\phi$.

The following proposition gathers some basic facts about eigenvectors.

\begin{proposition}\label{prop:properties}
  Let $\phi$ be a Lorentz transformation and $\z$ be a $\lambda$-eigenvector of $\phi$, then
  \begin{enumerate}
      \renewcommand{\labelenumi}{(\roman{enumi})}
      \renewcommand{\theenumi}{\labelenumi}
    \item\label{prop:complexconj} $\conjugate\z$ is an eigenvector of $\phi$ with eigenvalue $\conjugate\lambda$,
    \item\label{prop:powers} $\z$ is an eigenvector of $w^k$, $k\in\mathbb{N}$, with eigenvalue $\lambda^k$,
    \item\label{prop:inverse} $\z$ is an eigenvector of $w^{-1}$ with eigenvalue $\lambda^{-1}$,
    \item\label{prop:groupconj} let $\varphi\in O_\B(V)$, then $\varphi(\z)$ is an eigenvector of $\varphi\phi\varphi^{-1}$ with eigenvalue $\lambda$.
  \end{enumerate}
\end{proposition}

\begin{proposition}\label{prop:orthogonal}
  Let $\z_1$ and $\z_2$ be $\lambda$- and $\mu$-eigenvectors of $\phi\in O_{\B}(V)$, respectively. If $\lambda\conjugate\mu\ne 1$, then $\B(\z_1,\z_2)=0$.
\end{proposition}

\begin{proof}
  Since $w$ preserves the bilinear form, we have
  \begin{equation*}
    \B(\z_1,\z_2)=\B(\phi(\z_1),\phi(\z_2))=\lambda\conjugate\mu\B(\z_1,\z_2).
  \end{equation*}
  So $\B(\z_1,\z_2)=0$ because $\lambda\conjugate\mu\ne 1$.
\end{proof}

In the following propositions, we classify Lorentz transformations into three types. Such a classification is present in many references, often in the language of M\"obius transformations or hyperbolic isometries, see for instance \cite[Chapter 4, Theorem 1.6]{alekseevskij_geometry_1993}, \cite[Section~4.7]{ratcliffe_foundations_2006}, \cite[Proposition 4.5.1]{krammer_conjugacy_2009} and \cite[Section 7.8]{shafarevich_linear_2013}.  Our formulation is adapted from \cite[Chapter III]{riesz_clifford_1958}, which deals with Lorentz space and is suitable for our use.  See also discussions in~\cite[Section~3.3]{cecil_lie_2008} for a geometric insight.

\begin{proposition}[{\cite[Section~3.7]{riesz_clifford_1958}}]\label{prop:3types}
  Lorentz transformations are partitionned into three types:
  \begin{itemize}
    \item \Dfn{Elliptic transformations} are diagonalizable, and have only unimodular eigenvalues,
    \item \Dfn{Parabolic transformations} have only unimodular eigenvalues, but are not diagonalizable,
    \item \Dfn{Hyperbolic transformations} are diagonalizable and have exactly one pair of simple, real, non-unimodular eigenvalues, namely $\lambda^{\pm 1}$ for some $\lambda>1$.
  \end{itemize}
\end{proposition}

\begin{proposition}[{\cite[Section~3.7-3.9]{riesz_clifford_1958}}]\label{prop:hyperbolic}
  The two non-unimodular eigendirections of a hyperbolic transformation are light-like, while its unimodular eigenvectors are space-like. 
\end{proposition}

\begin{proposition}[{\cite[Section~3.10]{riesz_clifford_1958}}]\label{prop:parabolic}
  The Jordan form of a parabolic transformation $\phi$ contains a unique Jordan block of size~$3$, corresponding to the eigenvalue $\varepsilon=1$ or $-1$.  The $(n-2)$-dimensional \emph{real} subspace $U_\phi$ spanned by eigenvectors of $\phi$ is light-like. The $1$-dimensional light-like subspace of $U_\phi$ is a $\varepsilon$-eigendirection. The minimal polynomial $f(x)$ such that $f(\phi)$ annihilates~$U^\perp_\phi$ is $(x-\varepsilon)^2$.
\end{proposition}

\subsection{Spectral interpretation of limit roots}

Let $(W,S)_B$ be a Lorentzian Coxeter system. In this part, we consider the limit directions arising through sequences in the form of $(w^k)_{k\in\mathbb{N}}$ for some $w\in W$. We say that $w$ is an \Dfn{elliptic} (resp.~\Dfn{parabolic, hyperbolic}) element of $W$ if its corresponding transformation $\rho(w)$ is an elliptic (resp.~parabolic, hyperbolic) transformation.

\begin{theorem}\label{thm:elliptic}
  An element $w$ of a Lorentzian Coxeter system $(W,S)_B$ is of finite order if and only if $w$ is an elliptic Lorentz transformation.
\end{theorem}
\begin{proof}
  Assume that $w^k=e$ for some $k<\infty$. Since the minimal polynomial of $w$ divides $x^k-1$, its roots are all distinct and unimodular. Hence $w$ is diagonalizable with only unimodular eigenvalues, i.e.\ $w$ is elliptic.

  For the sake of contradiction, assume that $w$ is an elliptic element but of infinite order. Let the sequence $(w^k)_{k\in\mathbb{N}}$ act on a simple root $\alpha\in\Delta$. Since $w$ is diagonalizable and every eigenvalue is unimodular, we conclude that $\proj\alpha$ is itself a limit direction.  However, by \cite[Theorem 2.7]{hohlweg_asymptotical_2014}, limit directions arising from simple roots are on the light cone and therefore can not be a root.
\end{proof}

Consequently, whenever $w$ is of infinite order, it is either parabolic or hyperbolic.  Then Theorem~\ref{thm:main1} follows directly from the fact that the set of limit roots equals to the limit set of the Coxeter group regarded as a Kleinian group acting on the hyperbolic space~\cite[Theorem~1.1]{hohlweg_limit_2013}.  For the relation between fixed points and limit sets of Kleinian groups, see for instance~\cite[Lemma~2.4.1(ii)]{marden_outer_2007}.  

However, we provide here a different proof for the following reasons. Firstly, in the proof of Theorem~\ref{thm:parabolic} and~\ref{thm:hyperbolic}, the behavior of infinite-order elements is analysed in detail.  This will be useful in the proof of Theorem~\ref{thm:main2}. Secondly, our proof is independent of~\cite[Theorem~1.1]{hohlweg_limit_2013}.  In fact, as shown by Theorem~\ref{thm:main2} and Example~\ref{expl:nonLorentz}, limit directions arising from a space-like and a time-like base point are not necessarily the same.  So the equality between the set of limit roots and the limit set of Coxeter group is not trivial.  As mentioned in the introduction, our proof tries to rely minimally on hyperbolic geometry, in the hope of a deeper insight and further generalizations to non-Lorentzian infinite Coxeter systems.  In particular, concepts involved in the statement of Theorem~\ref{thm:main1} are all well defined for general infinite Coxeter systems.  The possibility of a generalization is discussed in detail in Section~\ref{ssec:nonLorentz}.

For a Lorentzian Coxeter group $W$, denote by $W_\infty$ the set of elements of infinite order, $W_\para$ the set of parabolic elements, and $W_\hyper$ the set of hyperbolic elements. Then
\[
  W_\infty=W_\para\sqcup W_\hyper.
\]
Given an element $w\in W_\infty$, the $(n-2)$-dimensional \emph{real} subspace spanned by unimodular eigenvectors of $w$ is called the \Dfn{unimodular subspace} of $w$, and denoted by $U_w$.

\begin{theorem}\label{thm:parabolic}
  The light-like eigendirection $\proj\x$ of a parabolic element $w\in W_\para$ is a limit root of~$W$.
\end{theorem}

\begin{figure}[!bhtp]
  \includegraphics{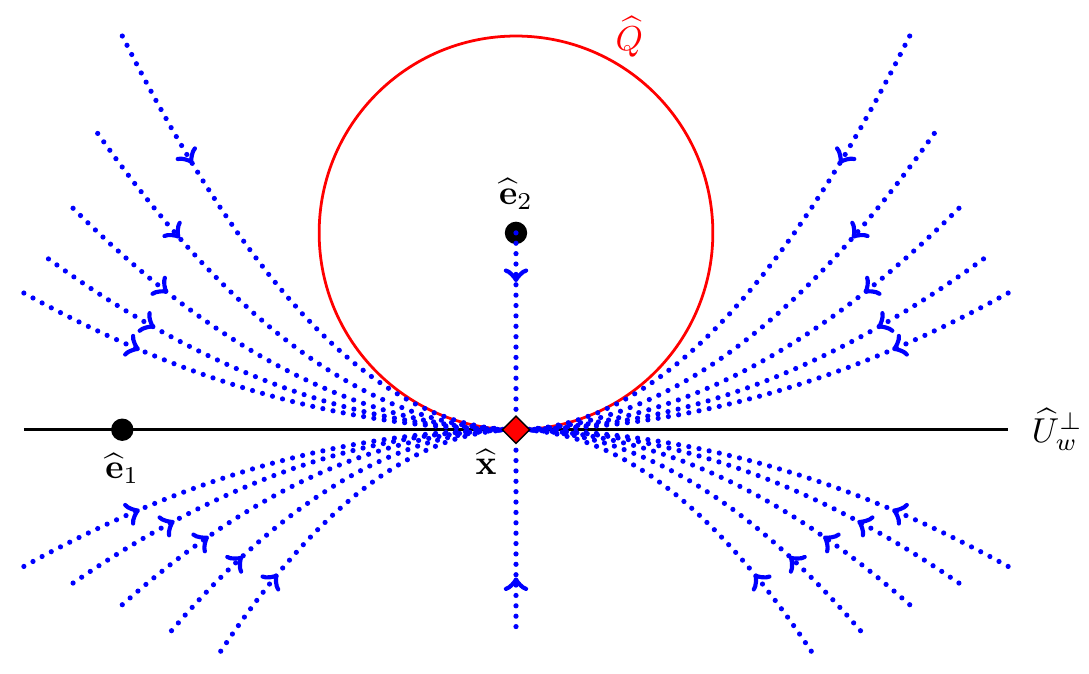}
  \caption{
    The dynamics of a parabolic element as described in the proof of Theorem~\ref{thm:parabolic}.
  }\label{fig:dynamic_para}
\end{figure}

\begin{proof}
  Let $\x$ be the light-like eigenvector of $w$ with eigenvalue $\varepsilon=1$ or $-1$. By Proposition \ref{prop:parabolic}, there are two vectors $\e_1, \e_2\notin U_w$ such that $w(\e_2)-\varepsilon\e_2=\e_1$ and $w(\e_1)-\varepsilon\e_1=\x$.  Note that~$\e_1$ is in $U^\perp_w$, while $\e_2$ is not. Let $\y\in V$ be a real vector, it can be decomposed into
  \begin{equation}\label{eqn:paradecomp}
    \y=a\x+b\e_1+c\e_2+\y^\circ
  \end{equation}
  where $\y^\circ\in U_w$ is orthogonal to $\e_1$, $\e_2$ and $\x$. Under the action of $w^k$, we have
  \begin{equation}\label{eqn:paraconverge}
    w^k(\y)=a_k\x+b_k\e_1+c_k\e_2+w^k(\y^\circ).
  \end{equation}
  where the coefficients
  \begin{align*}
    a_k&=a\varepsilon^k+bk\varepsilon^{k-1}+c\tbinom{k}{2}\varepsilon^{k-2},\\
    b_k&=b\varepsilon^k+ck\varepsilon^{k-1},\\
    c_k&=c\varepsilon^k.
  \end{align*}
  The term $w^k(\y^\circ)$ has bounded norm since $\y^\circ$ is contained in the unimodular subspace $U_w$. As long as $\y\ne U_w$, we have $b\ne 0$ or $c\ne 0$, then $b_k=o(a_k)$ and $c_k=o(b_k)$, i.e.\ $a_k$ dominates $b_k$ and $c_k$, and $b_k$ dominates $c_k$, as $k$ tends to infinity. Consequently, the direction of the sequence~$(w^k(\y))_{k\in\mathbb{N}}$ converges to the direction of $\x$. That is, $\proj\x$ is a limit direction arising from $\proj\y$ through the sequence~$(w^k)_{k\in\mathbb{N}}$.

  Since $\Delta$ spans the vector space $V$, there is a simple root $\alpha\notin U_w$. Using $\alpha$ as the base point, we conclude that $\proj\x$ is a limit root. The dynamics of a parabolic element is illustrated in Figure~\ref{fig:dynamic_para}. 
\end{proof}

\begin{theorem}\label{thm:hyperbolic}
  Let $w\in W_\hyper$ be a hyperbolic element. The two light-like eigendirections of $w$ are limit roots and the projective unimodular subspace $\proj U_w$ is contained in the set of limit directions of~$W$.
\end{theorem}

\begin{figure}[!htbp]
  \includegraphics{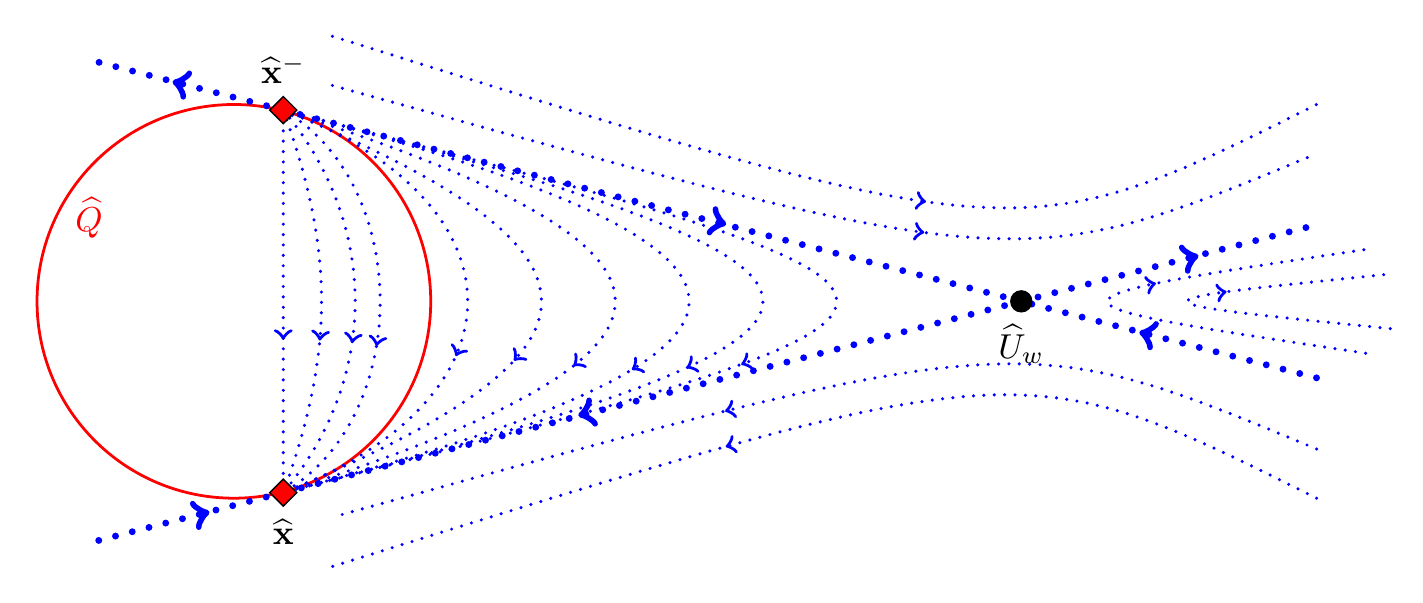}
  \caption{
    The dynamics of a hyperbolic element as described in the proof of Theorem~\ref{thm:hyperbolic}.
  }\label{fig:dynamic_hyper}
\end{figure}

\begin{proof}
  From Proposition~\ref{prop:3types}, the element $w$ possesses a light-like eigenvector $\x$ which is non-unimodular. By replacing $w$ with $w^{-1}$ if necessary, we may assume that the eigenvalue $\lambda$ corresponding to $\x$ is greater than $1$. Let $\x^-$ be an eigenvector of $w$ with eigenvalue $\lambda^{-1}$.  Since~$w$ is diagonalizable, there exists an eigenbasis of $V_\mathbb{C}$ consisting of $\x$, $\x^-$ and $n-2$ unimodular eigenvectors of $w$.  Then any real vector $\y\in V$ can be decomposed into
  \begin{equation}\label{eqn:hyperdecomp}
    \y=a^+\x + a^-\x^- +\y^\circ,
  \end{equation}
  where $\y^\circ\in U_w$ is orthogonal to $\x$ and $\x^-$. By Proposition~\ref{prop:properties}\ref{prop:powers}, we have 
  \[
    w^k(\y) =  a^+\lambda^k\x+a^-\lambda^{-k}\x^-+w^k(\y^\circ)
  \]
  Using $\y$ as the base point, three cases are possible
  \begin{description}
    \item[Case 1] $a^+\ne 0$, so $\y\notin H_{\x^-}$. The coefficient $a^+\lambda^k$ diverges to infinity while the coefficient~$a^-\lambda^{-k}$ tends to 0. Besides, the term $w^k(\y^\circ)$ has a bounded norm since $\y^\circ\in U_w$. Therefore, as $k$ tends to infinity, the direction of the sequence $(w^k(\y))_{k\in\mathbb{N}}$ converges to the direction of $\x$. So $\proj\x$ is a limit direction arising from $\proj\y$ through the sequence $(w^k)_{k\in\mathbb{N}}$.
    \item[Case 2] $a^+=0$ and $\y^\circ\ne 0$, so $\y\in H_{\x^-}$ but $\proj\y\ne\proj\x^-$. Again, the coefficient $a^-\lambda^{-k}$ tends to~$0$, so $\y^\circ\in U_w$ is an accumulation point of the set $\{w^k(\y)\mid k\in \mathbb{N}\}$. Consequently, $\proj\y^\circ\in  U_w$ is a limit direction arising from $\proj\y$ through some subsequence of $(w^k)_{k\in\mathbb{N}}$.
    \item[Case 3] $a^+=0$, $\y^\circ=0$ and $a^-\ne 0$, so $\proj\y=\proj\x^-$.  The sequence of vectors $(w^k(\y))$ converges to $0$, while the direction of $w^k\cdot\proj\y$ remains $\proj\x^-$ for all $k\in\mathbb{N}$. The sequence $(w^k\cdot\proj\x)$ visits only one point in $\mathbb{P}V$, so no limit direction arises.
  \end{description}
  Since $\Delta$ spans the vector space $V$, there is a simple root $\alpha\notin H_{\x^-}$. Using $\alpha$ as the base point, we conclude from Case $1$ that $\proj\x$ is a limit root.  The dynamics of a hyperbolic element is illustrated in Figure~\ref{fig:dynamic_hyper}. 
\end{proof}

For a Lorentzian Coxeter system $(W,S)_B$, let $E_\para$ (resp.~$E_\hyper$) be the set of light-like eigendirections of elements in $W_\para$ (resp.~$W_\hyper$). We now prove the following theorem, which is equivalent to Theorem \ref{thm:main1} since $E_\spec=E_\hyper\sqcup E_\para$. It can be derived from the minimality of $E_\Phi$ under the action of $W$~\cite[Theorem~3.1(b)]{dyer_imaginary2_2013}. We provide here a self-contained proof using the results in Section~\ref{ssec:lightlike}.

\begin{theorem}\label{thm:spectral} 
  For a Lorentzian Coxeter system $(W,S)_B$, the set of light-like eigenvectors $E_\para$ (if not empty) and $E_\hyper$ are dense in the set of limit roots $E_\Phi$.
\end{theorem}

\begin{proof}
  We first prove that $E_\hyper$ is not empty. Since the group is infinite, the set of limit root $E_\Phi$ is not empty. Let $\proj\x$ be a limit root, since $\Delta$ spans $V$, there is a simple root $\alpha\in\Delta$ such that $\B(\alpha,\x)\ne 0$, and the reflection in $\alpha$ gives a limit root $\proj\y$ different from $\proj\x$. Take two projective roots $\proj\alpha$ and $\proj\beta$ respectively close to $\proj\x$ and $\proj\y$, such that the segment $[\proj\alpha,\proj\beta]$ intersect the light cone at two points.  Then the product of the reflections in $\alpha$ and $\beta$ give an hyperbolic element in $W$.

  Let $\proj\gamma \in E_\Phi$ be a limit root obtained from an injective sequence $(\proj\gamma_i)_{i\in\mathbb{N}}$ of projective roots. By passing to a subsequence, we may assume that $\proj\gamma$ is obtained from an injective sequence $(g_i(\alpha))_{i\in\mathbb{N}}$, where~$\alpha$ is a fixed simple root in $\Delta$ and $g_i\in W$.

  Let $\proj\z\in E_\hyper$ be a light-like eigendirection of a hyperbolic element $w\in W_\hyper$. By Proposition~\ref{prop:properties}\ref{prop:groupconj}, $g_k\cdot\proj\z$ is a light-like eigendirection of the hyperbolic element $g_kwg_k^{-1}$. So the sequence $(g_k\cdot\proj\z)_{k\in\mathbb{N}}$ is a sequence of limit roots in $E_\hyper$. By compactness, we may assume that $(g_k\cdot\proj\z)$ converges.  By Theorem~\ref{thm:lightlike}, its limit is a limit root.

  The last step consists of applying Corollary~\ref{cor:samelimits} to the sequence $(g_k\cdot\proj\z)_{k\in\mathbb{N}}$ to prove that $\proj\gamma$ is the limit. For this, it remains to prove that the sequence contains infinitely distinct points. The Coxeter graph of a Lorentzian Coxeter system consists of irreducible finite Coxeter graphs and one irreducible Coxeter graph of Lorentzian type. The geometric representations of finite Coxeter group and irreducible Lorentzian Coxeter groups act irreducibly on the space, see \cite[Proposition~6.3]{humphreys_reflection_1992} and \cite[Lemma~14]{vinberg_discrete_1971}. Using this irreducible action, we can find a basis $\b_1,\dots,\b_n$ of~$V$ in the orbit $W\cdot \proj\z$. By the injectivity of the sequence $g_i$, there exists a basis vector $\b_i$, such that $g_k\cdot \proj\b_i$ is injective. Taking this $\b_i$ as the vector $\z$ above finishes the proof.

  The same arguments work, mutatis mutandis, for parabolic elements.
\end{proof}

\begin{figure}[!hbt]
  \centering
  \scalebox{0.86}{
    \begin{subfigure}[b]{.45\textwidth}
      \centering
      \includegraphics{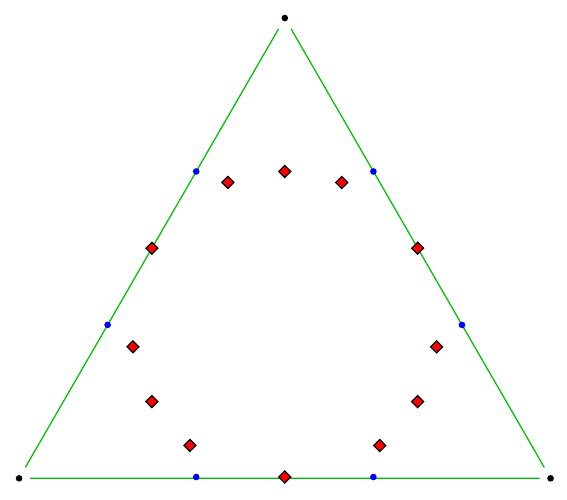}
      \caption{
	The 12 light-like eigendirections of parabolic elements of length smaller or equal to $6$.
      }\label{fig:projroots}
    \end{subfigure}
    \qquad
    \begin{subfigure}[b]{.45\textwidth}
      \centering
      \includegraphics{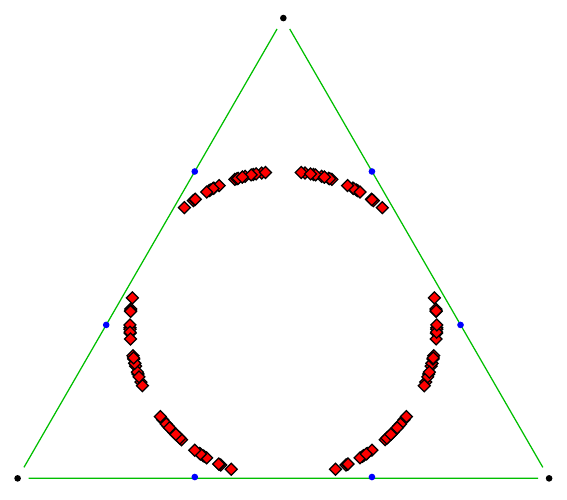}
      \caption{
	The 126 light-like eigendirections of hyperbolic elements of length smaller or equal to $6$.
      }\label{fig:projweights}
    \end{subfigure}
  }
  \caption{
    Light-like eigendirections of some elements of the classic universal geometric Coxeter group of rank 3, seen in the affine space spanned by the simple roots.
  }\label{fig:dense_limitroots}
\end{figure}

Examples of light-like eigendirections of parabolic and hyperbolic elements are illustrated in Figure~\ref{fig:dense_limitroots} for the rank 3 universal Coxeter group with the classical geometric representation. Observe that their number and distribution are quite different.

\subsection{Non-Lorentzian Coxeter systems} \label{ssec:nonLorentz}

The density of $E_\hyper$ in $E_\Phi$ can also be derived from a result of Conze and Guivarc'h~\cite{conze_limit_2000}. In that paper, a linear transformation is said to be \Dfn{proximal} if there exists a \emph{simple} eigenvalue which is strictly greater than any other eigenvalue in absolute value. This is the case for hyperbolic transformations of Lorentzian Coxeter systems. An eigenvector with this eigenvalue is called \Dfn{dominant eigenvector}, so $E_\hyper$ is in fact the set of dominant eigendirections. By \cite[Proposition 2.4]{conze_limit_2000}, $\closure{E_\hyper}$ is the only minimal set of the action of $W$. Then by \cite[Theorem~3.1(b)]{dyer_imaginary2_2013}, this set is $E_\Phi$. 

For a non-Lorentzian infinite Coxeter system $(W,S)$, it remains true that the set of dominant eigendirections is dense in the set of limit roots.  It is proved in~\cite[Section~6.5]{krammer_conjugacy_2009} that an element $w\in W$ is proximal if it is not contained in any proper parabolic subgroup of $W$.

However, an arbitrary element $w$ of infinite order is not necessarily proximal, even if $w$ has non-unimodular eigenvalues. Let $\lambda$ denotes the eigenvalue of $w$ with largest absolute value, it is possible that the geometric multiplicity of $\lambda$ is not $1$.  The real subspace $U$ spanned by $\lambda$-eigenvectors of $w$ is \Dfn{totally isotropic}, meaning that $\B(\x,\y)=0$ for any $\x,\y\in U$. Then any direction $\proj\x$ in $\proj U$ is an isotropic limit direction, but only those in $\convex(\Delta)$ are limit roots. Therefore, Theorem~\ref{thm:main1} and~\ref{thm:lightlike} do not generalize to other infinite Coxeter systems.

\begin{example}\label{expl:nonLorentz}
  In Figure~\ref{fig:nonLorentz}, we show an example inspired by \cite[Example 5.8]{hohlweg_asymptotical_2014} and \cite[Example~7.12 and~9.18]{dyer_imaginary_2013}. It is the Coxeter graph of an irreducible Coxeter group, associated with a non-degenerate bilinear form of signature $(3,2)$. We adopt Vinberg's convention for Coxeter graphs to encode both the Coxeter system $(W,S)$ and the matrix $B$. The element $s_1s_2s_4s_5$ is of infinite order with a simple eigenvalue $1$, and two non-unimodular eigenvalues $7\pm4\sqrt 3$, each with multiplicity~$2$. Therefore, its decomposition into eigenspaces do not fall into any of the types shown in Proposition~~\ref{prop:3types}.  \end{example}

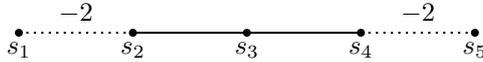
\begin{figure}[!htp]
  \centering
  \input{nonLorentz.tex}
  \caption{
    The Coxeter graph of a non-Lorentzian irreducible Coxeter group for which the limit roots are not the only limit directions on the isotropic cone.
  }\label{fig:nonLorentz}
\end{figure}

%
%

\section{Coxeter arrangement and limit directions}\label{sec:Arrangement}

In this section, we characterize the set of limit directions of a Lorentzian Coxeter system in terms of the Coxeter arrangement in projective space.

\subsection{Projective Lorentzian Coxeter arrangement}

Let $(W,S)_B$ be a Lorentzian Coxeter system and $(\Phi,\Delta)$ be the associated root system. The hyperplane $H_\gamma$ orthogonal to a positive root $\gamma\in\Phi^+$ is time-like and fixed by the reflection $\sigma_\gamma$. The \Dfn{projective Lorentzian Coxeter arrangement} is the set of projective hyperplanes \[\H=\{\proj H_\gamma\mid\gamma\in\Phi^+\}.\] Clearly, $\H$ is invariant under the action of $W$.

Let $\gamma\in\Phi^+$ be a positive root, the \Dfn{sign map} $\sign_\gamma:\mathbb{P}V\to\{+,-,0\}$ sends $\proj\x\in\mathbb{P}V$ to the sign of $\B(\x,\gamma)$. To every direction $\proj\x\in\mathbb{P}V$ is associated a \Dfn{sign sequence} $(\sign_\gamma(\proj\x))_{\gamma\in\Phi^+}$ indexed by the positive roots. The set of points with a fixed sign sequence is called a \Dfn{cell}. The projective space~$\mathbb{P}V$ is therefore decomposed into cells. The set of cells is denoted by $\Sigma$. For two cells $C,C'\in\Sigma$, if the sign sequence of $C'$ is obtained from the sign sequence of $C$ by changing zero or more (maybe infinitely many) $+$'s or $-$'s to $0$, we say that $C'$ is a \Dfn{face} of $C$, and write $C'\le C$.  This defines a partial order on $\Sigma$. The \Dfn{support} of a cell $C$ is defined as \[ \support(C)=\bigcap_{\substack{\gamma\in\Phi^+\\ \sign_\gamma(C)=0}}\proj H_\gamma \] Cells are open in their supports. The dimension of a cell is defined as the dimension of its support. The codimension of a cell is defined similarly. A cell is said to be \Dfn{space-like} (resp.~\Dfn{light-like}, \Dfn{time-like}) if its support is a projective space-like (resp.~light-like, time-like) subspace. Cells with no $0$ in their sign sequences are called \Dfn{chambers}, they are connected components of the complement $\mathbb{P}V\setminus\cup_{\gamma\in\Phi^+}\proj H_\gamma$. Cells with exactly one $0$ in their sign sequences are called \Dfn{panels}, they are time-like codimensional~$1$ cells. The \Dfn{projective Tits cone} $\T$ is the union of cells whose sign sequences contain finitely many~$-$'s. In the Lorentzian case, the Tits cone is the convex cone over the set of weights $\Omega$, and contains the light cone \cite[Corollary~1.3]{maxwell_sphere_1982}.

\begin{remark}
  A Lorentzian hyperplane arrangement is infinite and not discrete.  Unlike finite hyperplane arrangements, the union of hyperplanes in a Lorentzian hyperplane arrangement is in general not a closed set.
\end{remark}

\subsection{Unimodular subspaces}

Let $\alpha,\beta\in\Phi$ be two positive roots. If the segment $[\proj\alpha,\proj\beta]$ intersect the projective light cone $\proj Q$ transversally (i.e.\ $\B(\alpha,\beta)<-1$), then $\sigma_\alpha\sigma_\beta$ is a hyperbolic transformation. If the segment $[\proj\alpha,\proj\beta]$ is tangent to $\proj Q$ (i.e.\ $\B(\alpha,\beta)=-1$), then $\sigma_\alpha\sigma_\beta$ is a parabolic transformation. In either case, we know from \cite[Section 4]{hohlweg_asymptotical_2014} that the limit roots of the subgroup generated by $\sigma_\alpha$ and $\sigma_\beta$ are the points in $\proj Q\cap[\proj\alpha,\proj\beta]$. By Theorem \ref{thm:hyperbolic}, these are the light-like eigendirections of $\sigma_\alpha\sigma_\beta\in W_\infty$.  The unimodular subspace of $\sigma_\alpha\sigma_\beta$ is clearly $H_\alpha\cap H_\beta$. We define
\[
  \L_\hyper=\bigcup_{\substack{\alpha,\beta\in\Phi^+\\\B(\alpha,\beta)<-1}}\proj H_\alpha\cap\proj H_\beta,\qquad
  \L_\para=\bigcup_{\substack{\alpha,\beta\in\Phi^+\\\B(\alpha,\beta)=-1}}\proj H_\alpha\cap\proj H_\beta,\qquad
\]

Furthermore, we define the unions of projective unimodular subspaces for parabolic, hyperbolic, and infinite-order elements.
\[
  \U_\para=\bigcup_{w\in W_\para}\proj U_w,\qquad
  \U_\hyper=\bigcup_{w\in W_\hyper}\proj U_w,\qquad
  \U=\U_\para\sqcup\U_\hyper=\bigcup_{w\in W_\infty}\proj U_w
\]
We have clearly $\L_\hyper\subset\U_\hyper$ and $\L_\para\subset\U_\para$. The following theorem concerns a reversed inclusion.

\begin{theorem}\label{thm:UinL}
  The projective unimodular subspace~$\proj U_w$ of an element of infinite order $w\in W_\infty$ is included in $\closure{\L_\hyper}$. In other words,
  \[
    \U\subseteq\closure{\L_\hyper}.
  \]
\end{theorem}

\begin{proof}
  Let $\Lambda$ be the continuous map that sends a pair of light-like directions $(\proj\x,\proj\y)\in\proj Q$ to the codimension-$2$ projective subspace $\proj H_\x\cap\proj H_\y$.

  For a hyperbolic element $w\in W_\hyper$, let $(\proj\x,\proj\y)$ be its two non-unimodular eigendirections, then $\Lambda(\proj\x,\proj\y)=\proj U_w$.  Since $\proj\x$ and $\proj\y$ are limit roots, we can find two sequences of roots $(\alpha_k)_{k\in\mathbb{N}}$ and $(\beta_k)_{k\in\mathbb{N}}$ such that $\proj\alpha_k$ converges to $\proj\x_w$ and $\proj\beta_k$ converges to $\proj\y$. The sequence of segments $[\proj\alpha_k,\proj\beta_k]$ eventually intersect the projective light cone $\proj Q$ at two limit roots, say $\proj\x_k$ and $\proj\y_k$. These two limit roots determine a unimodular projective subspace $\proj U_k=\Lambda(\proj\x_k,\proj\y_k)\in\L_\hyper$. The two sequences of limit roots $(\proj\x_k)_{k\in\mathbb{N}}$ and $(\proj\y_k)_{k\in\mathbb{N}}$ converge to $\proj\x$ and $\proj\y$ respectively. By the continuity of $\Lambda$, $\proj U_k$ converges to $\proj U_w$ as $k$ tends to infinity. See Figure~\ref{fig:unimod_hyper} for an illustration.

  \begin{figure}[!htbp]
    \includegraphics{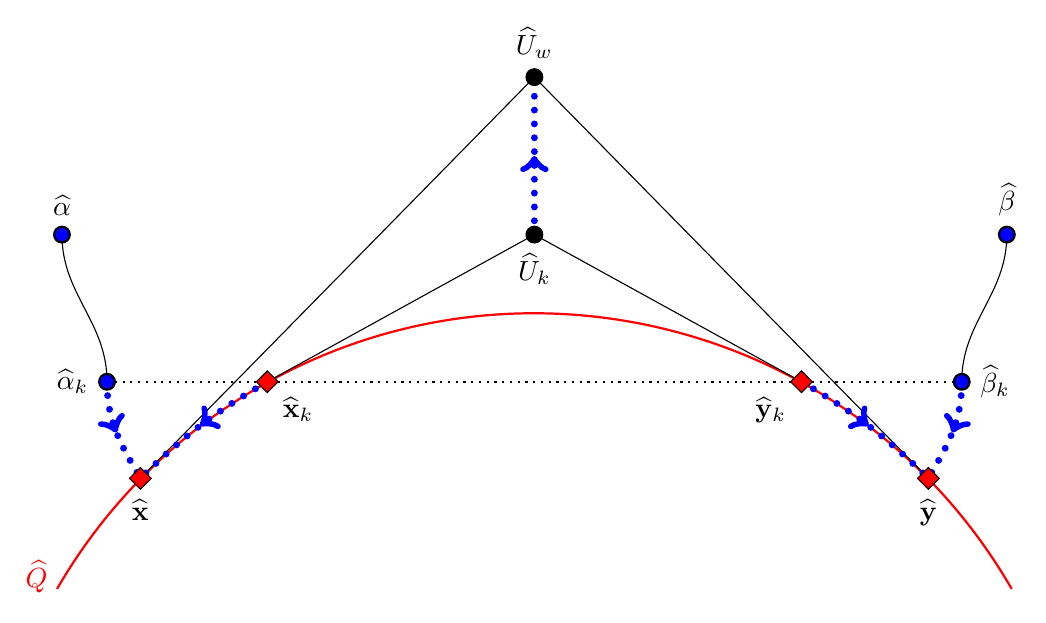}
    \caption{
      Illustration for the proof of Theorem~\ref{thm:UinL} in the case of hyperbolic elements at step $k$.
    }\label{fig:unimod_hyper}
  \end{figure}

  For a parabolic element $w\in W_\para$, let $\proj\x\in\proj U_w$ be the light-like eigendirection of $w$, then $\proj U_w+\proj U^\perp_w=\proj H_\x$ is the codimension-$1$ hyperplane that is tangent to $\proj Q$ at $\proj\x$.  See Figure~\ref{fig:unimod_para} for an illustration. Let $W_\x\ni w$ be the stabilizer subgroup of $\proj\x$.  It is generated by reflections through positive roots~\cite[Lemma~1.10]{dyer_imaginary_2013}
  \[
    \Phi_\x=\{\gamma\in\Phi^+\mid\B(\gamma,\proj\x)=0\}\subset H_\x.
  \]
  Since the restriction of $\B$ on $H_\x$ is positive semi-definite with a radical of dimension 1, $W_\x$ is an irreducible affine Coxeter group, and $\proj\x$ is the only limit root, see~\cite[Corollary~2.16]{hohlweg_asymptotical_2014}. Furthermore, we claim that $\Phi_\x\not\subset U_w$, otherwise we have $U^\perp_w\subseteq U_w$ and $H_\x\subseteq U_w$, which is not possible.

  Since $\proj\x\in\convex(\proj\Phi_\x)$, we can find two positive roots $\alpha,\beta\in\Phi_\x$ with opposite sign for the coefficient $b$ in the decomposition \eqref{eqn:paradecomp}. Recall that, under the action of $(w^k)_{k\in\mathbb{N}}$, the coefficients~$b_k$ dominates $c_k$ in Equation \eqref{eqn:paraconverge} as $k$ tends to infinity. Therefore, in the projective space, as the sequences $(\proj\alpha_k=w^k\cdot\proj\alpha)_{k\in\mathbb{N}}$ and $(\proj\beta_k=w^k\cdot\proj\beta)_{k\in\mathbb{N}}$ converge to $\proj\x$, they are also \Dfn{asymptotically tangent} to $\proj U^\perp_w$ at $\proj\x$, in the sense that $\lim_{k\to\infty}(\proj\alpha_k-\proj\alpha_{k-1})\in\proj U^\perp_w$.

  Since the simple roots $\Delta$ spans $V$ and the Coxeter group $W$ is Lorentzian, there is a simple root $\delta$ such that the segment $[\proj\delta,\proj\x]$ intersects $\proj Q$ at two points. We now construct two sequences of projective roots $(\proj\alpha'_k=w^k\sigma_\alpha\cdot\proj\delta)_{k\in\mathbb{N}}$ and $(\proj\beta'_k=w^k\sigma_\beta\cdot\proj\delta)_{k\in\mathbb{N}}$. Both sequences converge to~$\proj\x$ and are asymptotically tangent to $\proj U^\perp_w$ at $\proj\x$.  The sequence of segments $[\proj\alpha'_k,\proj\beta'_k]$ eventually intersect~$\proj Q$ at two limit roots, say $\proj\x_k$ and $\proj\y_k$, which determine a unimodular projective subspace $U_k=\Lambda(\proj\x_k,\proj\y_k)\in\L_\hyper$. The two sequences of limit roots $(\proj\x_k)_{k\in\mathbb{N}}$ and $(\proj\y_k)_{k\in\mathbb{N}}$ both converge to $\proj\x$, and are asymptotically tangent to~$\proj U^\perp_w$ at $\proj\x$. By the continuity of $\Lambda$, $\proj U_k$ converges to the orthogonal companion of $\proj U^\perp_w$, which is~$\proj U^{\perp\perp}_w=\proj U_w$, as $k$ tends to infinity.

  \begin{figure}[!htbp]
    \includegraphics{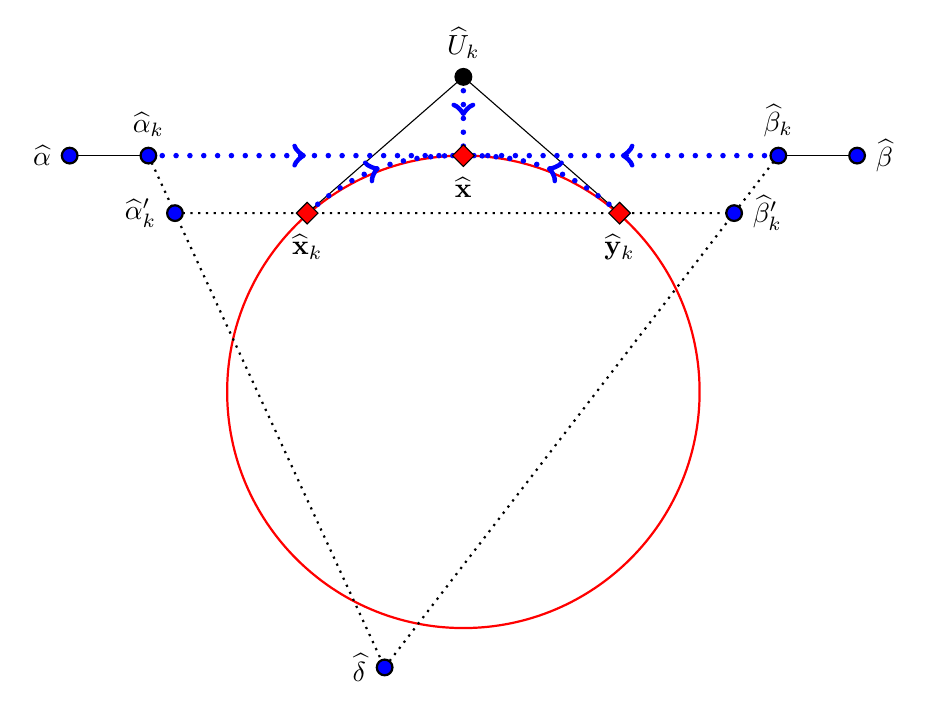}
    \caption{
      Illustration for the proof of Theorem~\ref{thm:UinL} in the case of parabolic elements at step $k$.
    }\label{fig:unimod_para}
  \end{figure}
\end{proof}

\subsection{Limit directions}

By Theorem \ref{thm:hyperbolic}, we have $\L_\hyper\subset \U_\hyper\subset E_V$.  Notably, while limit roots arise from projective roots and projective weights, it is possible for a projective root to be a limit direction, and space-like projective weights are all limit directions. 

In this part, we prove the other inclusion of Theorem \ref{thm:main2}, namely that $E_V\subseteq\closure{\L_\hyper}$. We will need the following two lemmas.

\begin{lemma}[{Selberg's lemma, \cite[Section~7.5]{ratcliffe_foundations_2006}}]\label{lem:selberg}
  Every finitely generated subgroup $G$ of $GL(n,\mathbb{C})$ has a torsion-free normal subgroup of finite index.
\end{lemma}

\begin{lemma}\label{lem:finitind}
  Let $G$ be a group acting on a vector space $X$, and $H$ be a subgroup of $G$ of finite index. Then the set of limit points of $H$ is equal to the set of
  limit points of $G$. 
\end{lemma}

\begin{proof}
  A limit point of $H$ is a limit point of $G$. Conversely, let $x\in X$ be a limit point of $G$ arising from a base point $x_0\in X$ through the sequence $(g_k)_{k\in\mathbb{N}}\in G$. Since $H$ is of finite index, by passing to a subsequence if necessary, we may assume that the sequence $(g_k)_{k\in\mathbb{N}}$ is contained in a single coset of $H$. That is, there is a sequence $(h_k)_{k\in\mathbb{N}}\in H$ and a fixed element $g\in G$ such that $g_k=h_kg$ for all $k\in\mathbb{N}$. Then $x$ is a limit point of $H$ arising from the point $g(x_0)$ through the sequence $(h_k)_{k\in\mathbb{N}}$.
\end{proof}

\begin{theorem}\label{thm:EinL}
  The set of limit direction of a Lorentzian Coxeter system is included in $\closure{\L_\hyper}$,
  \[
    E_V\subseteq\closure{\L_\hyper}.
  \]
\end{theorem}

\begin{proof}
  Assume that some $\proj\x\notin\closure{\L_\hyper}$ is a limit direction. By Selberg's lemma, there exists a subgroup of $W$ of finite index whose only element of finite order is the identity. By Lemma~\ref{lem:finitind}, the limit direction $\proj\x$ arises through a sequence of infinite-order elements. By the definition of limit point, for any neighborhood $N$ of $\proj\x$, there is infinitely many elements $w\in W_\infty$ such that $(w\cdot N)\cap N\ne\varnothing$.

  We claim that $\proj\x$ is not in the Tits cone~$\T$. If $\proj\x$ is in the interior of $\T$, there is a neighborhood~$N$ of $\proj\x$ such that $(w\cdot N)\cap N=\varnothing$ for any element $w\in W_\infty$, see \cite[Exercise 2.90]{abramenko_buildings_2008}. If $\proj\x$ is on the boundary of $\T$, the stabilizer of $\proj\x$ is infinite, so there is an element $w\in W_\infty$ such that $\proj\x\in\proj U_w\subset\closure{\L_\hyper}$, contradicting our assumption. Let $C\in\Sigma$ be the cell of $\H$ containing $\proj\x$. Since $\proj\x\notin\closure{\L_\hyper}$, $C$ does not intersect $U_w$ for any $w\in W_\infty$. We now prove that $(w\cdot C)\cap C=\varnothing$ for any element $w\in W_\infty$. 

  Assume that $w\in W_\infty$ is an element of infinite order such that $(w\cdot C)\cap C\ne\varnothing$. Since $\H$ is invariant under the action of $W$, we must have $w^k\cdot C=C$ for any integer $k\in\mathbb{Z}$. Since $C\cap\proj U_w=\varnothing$, some vertices of $C$ is not in $\proj U_w$. Let $v\in\closure{C}$ be such a vertex. From the proof of Theorem~\ref{thm:parabolic} and~\ref{thm:hyperbolic}, we see that the sequence $(w^k(v))_{k\in\mathbb{N}}$ converges to a light-like eigendirection of $w$. In order for~$C$ to be invariant under the action of $w$, the vertex $v$ must be a light-like eigendirection of $w$. The element~$w$ must be hyperbolic, otherwise if $w$ is parabolic, we have $C\subset\proj U_w\subset\closure{\L_\hyper}$. Furthermore,~$v$ is the only vertex of $C$ that is not in $\proj U_w$. Otherwise, let $u\notin\proj U_w$ be another vertex of $C$, then the segment $[u,v]$ is inside the light cone $\proj Q$, and $\proj\x\in C$ is in the interior of the Tits cone $\T$. Therefore, the cell $C\subset H_v$ is light-like. 

  Let $W_C$ be the stabilizer subgroup of $C$. It is generated by reflections in the positive roots $\Phi_C=\{\gamma\in\Phi^+\mid C\subset H_\gamma\}$.  These roots lie on the orthogonal companion of $\support(C)$, which is tangent to the light cone $\proj Q$ at $v$. The restriction of the bilinear form $\B$ on the subspace spanned by $\Phi_C$ is positive semi-definite, so $W_C$ is an affine Coxeter group. We then conclude that there is an element $w'\in W_\para$ such that $C\in\proj U_{w'}\subset\closure{\L_\hyper}$, contradicting our assumption.

  We have proved that $(w(C))\cap C=\varnothing$ for all $w\in W_\infty$. Since $\proj\x\in C$ and $C$ is open in $\support(C)$, we can find a neighborhood $N$ of $\proj\x$ such that $(w\cdot N)\cap N=\varnothing$ for all $w\in W_\infty$. Therefore, $\proj\x$ can not be a limit direction.
\end{proof}

\subsection{Open problems on limit directions}\label{ssec:problem}

We proved that the set $E_V$ of limit directions is located between the set~$\L_\hyper$ and its closure. In fact, by Theorem \ref{thm:hyperbolic} and Section \ref{thm:UinL}, a stronger result can be obtained:
\[
  \U_\hyper\sqcup E_\Phi\subseteq E_V\subseteq\closure{\U_\hyper}.
\]

\begin{figure}[!htbp]
    \includegraphics{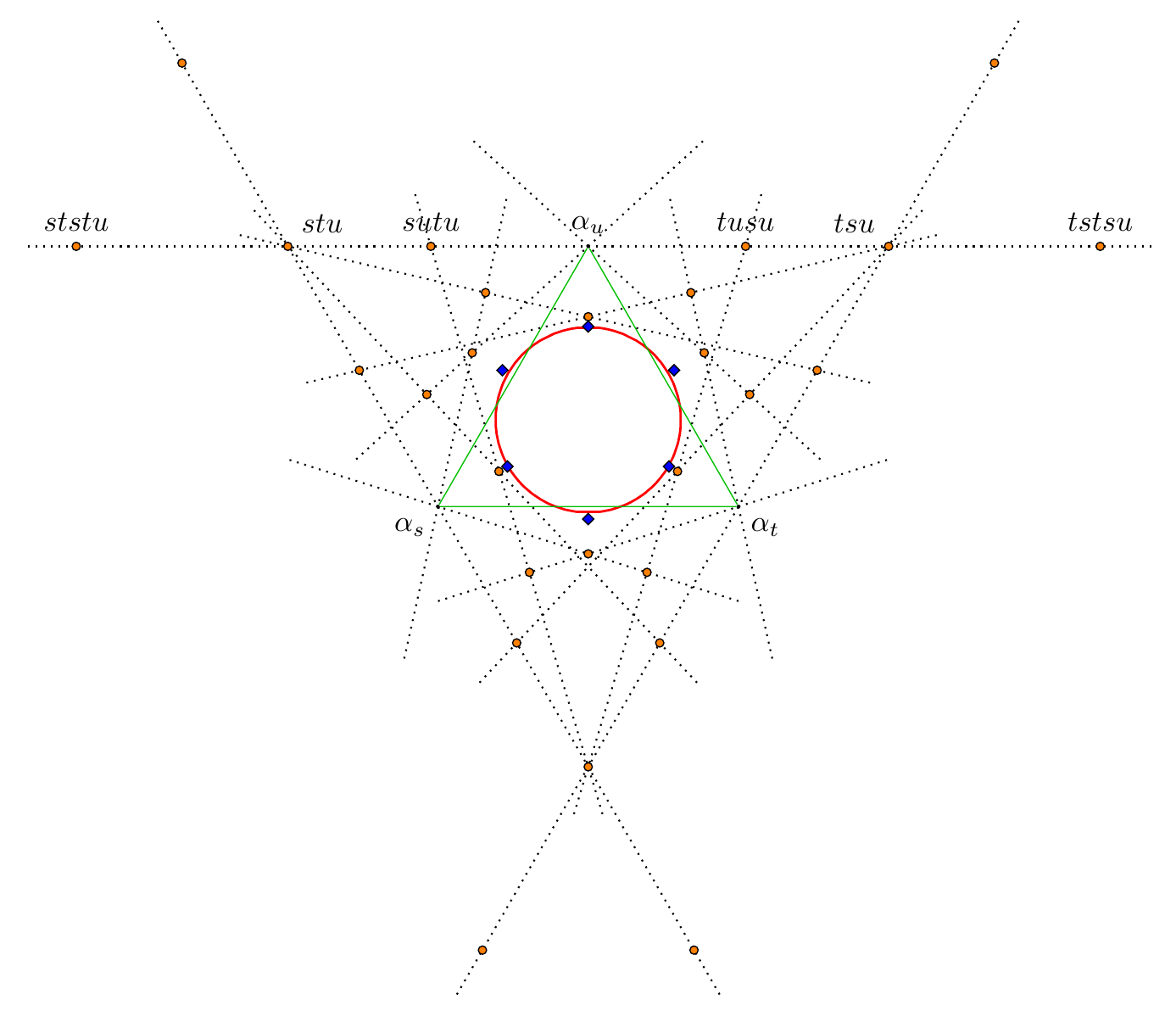}
  \caption{
    The affine representation of some space-like limit directions of the universal Coxeter group of rank $3$.  The associated bilinear form has $c_{ij}=1.1$ whenever $i\ne j$, same as in Figure~\ref{fig:intersection}.  The generators of the Coxeter group are $s$, $t$ and $u$, while $\alpha_s,\alpha_t$ and $\alpha_u$ in the figure are the corresponding simple roots.  The dots are unimodular eigenvectors of hyperbolic elements of length $\le 5$.  The six eigenvectors marked with diamonds are also weights.  Some unimodular eigenspaces are labeled by the corresponding hyperbolic elements. The dotted lines show some linear dependences between limit directions verified with Sage.  
  }\label{fig:limitdirection}
\end{figure}

Figure~\ref{fig:limitdirection} shows some unimodular eigenvectors for infinite order elements for the same geometric Coxeter system as in Figure~\ref{fig:intersection}.  Many of the eigenvectors are not in the intersections of the Coxeter arrangement, but can be approximated by intersections.  Interestingly, the eigenvectors seem to obey certain linear dependences which are not present in the hyperplane arrangement. For example, the unimodular eigenspaces of the elements $ststu, stu, sutu, tusu, tsu$ and $tstsu$ lie on a hyperplane which is not a reflecting hyperplane. It would be interesting to study these linear dependences of unimodular eigenvectors in relation with the structure of the Coxeter group.

\begin{problem}
  Prove or disprove the following equalities.
  \begin{align}
    E_V&=\U_\hyper\sqcup E_\Phi\label{eqn:equality1},\\
    E_V&=\closure{\U_\hyper}\label{eqn:equality2}.
  \end{align}
\end{problem}

Only one of the equality may be true. In fact, in the case where $E_\Phi=\proj Q$ (see for instance Figure~5 and Figure~8 of \cite{hohlweg_asymptotical_2014} and Figure~\ref{fig:dense_limitroots} of the present paper), $\U_\hyper$ is the union of countably many codimension-$2$ subspaces, while $\closure{\U_\hyper}$ consists of all space-like and light-like directions, so~$\U_\hyper$ is a proper subset of $\closure{\U_\hyper}$. Then Equation~\eqref{eqn:equality2} would imply counterintuitively that every non-time-like direction is a limit direction.

In the boundary $\partial(\U_\hyper)=\closure{\U_\hyper}\setminus\U_\hyper$, we know that the limit roots $E_\Phi\subset\partial(\U_\hyper)$ are limit directions. Since every limit direction in $\U_\hyper$ arises through a sequence of the form $(w^k)_{k\in\mathbb{N}}$ for $w\in W_\infty$, Equation \eqref{eqn:equality1} is equivalent to the following conjecture

\begin{conjecture}
  Every space-like limit direction arises through a sequence of the form $(w^k)_{k\in\mathbb{N}}$ for $w\in W_\infty$.
\end{conjecture}

To prove Equation~\eqref{eqn:equality2}, it suffices to prove that $E_V$ is closed. Let $E_\x$ be the set of limit directions arising from a fixed base point $\proj\x\in\mathbb{P}V$. Since $E_\x$ is the set of accumulation points of $W\cdot\proj\x$, it is clear that $E_V$ is closed and invariant under the action of $W$. In particular, $E_\x=E_\y$ if $\proj\x$ and $\proj\y$ lie in a same orbit of $W$, i.e.\ $\proj\x=w\cdot\proj\y$ for some $w\in W$. However, the set of limit directions, being the infinite union
\[
  E_V=\bigcup_{\proj\x\in\mathbb{P}V}E_\x=\bigcup_{\proj\x\in\mathbb{P}V/G}E_\x,
\]
may not be closed in general.

Since the set of limit roots $E_\Phi$ is a minimal set of~$W$, we have $E_\Phi\subseteq E_\x$ for all $\proj\x\in\mathbb{P}V$. In Section \ref{sec:Coxeter}, we have seen that $E_\x=E_\Phi$ if $\proj\x$ is a time-like, light-like, projective root or a projective weight. Using the argument in the proof of Theorem \ref{thm:EinL}, we can prove that $E_\x=E_\Phi$ if $\proj\x$ is in the Tits cone $\T$. Define the \emph{catchment set} $F_\Phi=\{\proj\x\in\mathbb{P}V\mid E_\x=E_\Phi\}$ of limit roots.

\begin{problem}
  Is $F_\Phi$ a closed set? 
\end{problem}

%
%

\bibliographystyle{amsalpha}
\bibliography{biblio}

\end{document}

%% file: limroot1.tex
\includegraphics[width=\textwidth]{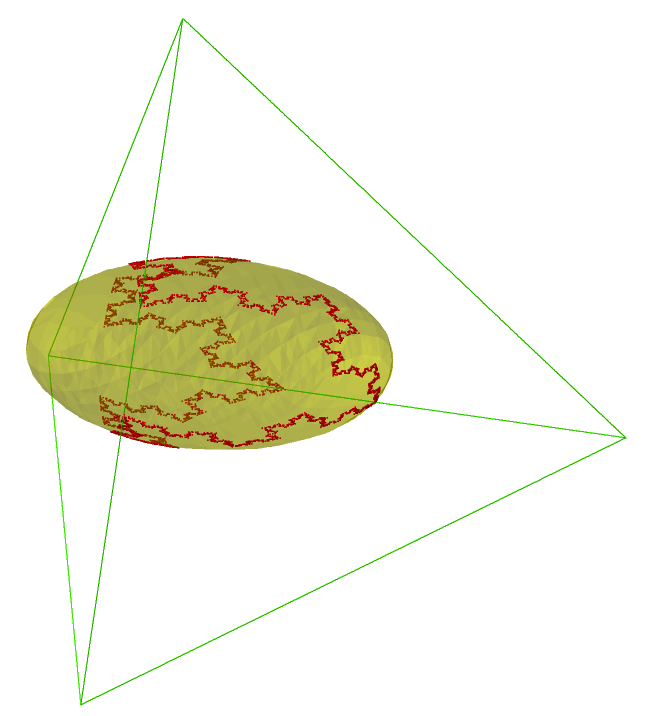}
\begin{tikzpicture}[thick, scale=1]
	\coordinate (s1) at (-2,0);
	\coordinate (s2) at (-1,0);
	\coordinate (s3) at (0,0);
	\coordinate (s4) at (1,0);

	\foreach \k in {1,...,4} \fill {(s\k)} circle (1pt);
	\draw[dotted] (s1) --node[above]{$-1.05$} (s2);
	\draw (s2)--(s3);
	\draw[dotted] (s3) --node[above]{$-1.05$} (s4);
\end{tikzpicture}

%% file: limroot2.tex
\includegraphics[width=\textwidth]{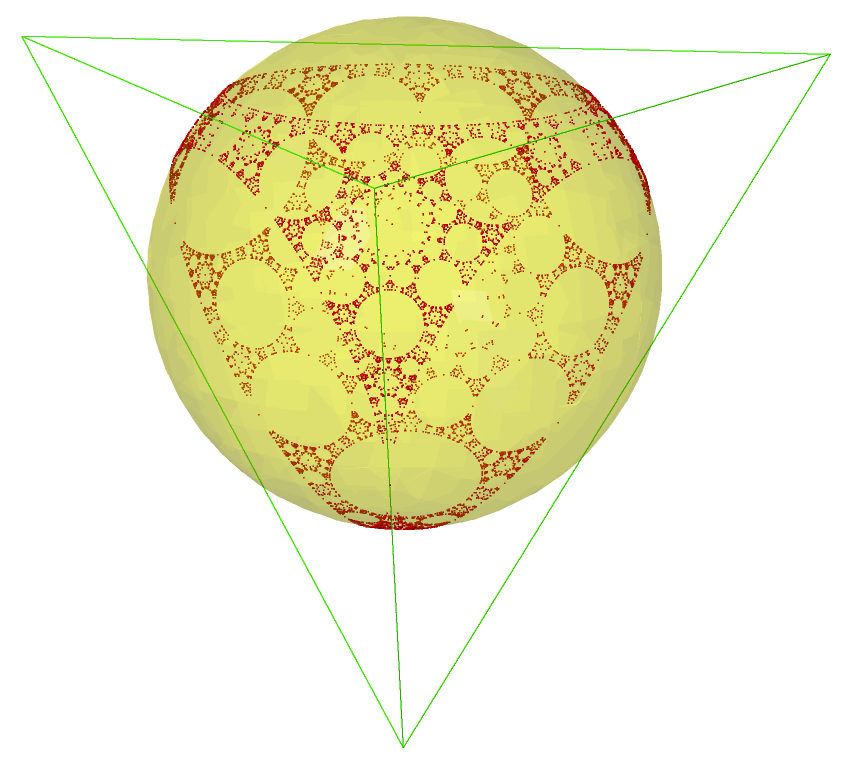}
\begin{tikzpicture}[thick, scale=1]
	\coordinate (s2) at (30:1);
	\coordinate (s3) at (150:1);
	\coordinate (s4) at (270:1);
	\coordinate (s1) at (0,0);

	\foreach \k in {1,...,4} \fill {(s\k)} circle (1pt);

	\draw (s1) -- (s4);
	\draw (s2)--(s3)--(s4)--(s2);
	\draw (s1) -- (s2);
	\draw (s1) -- (s3);
	\fill[white] (150:0.5) circle (5pt);
	\fill[white] (270:0.5) circle (5pt);
	\fill[white] (30:0.5) circle (5pt);

	\node at (150:0.5) {$5$};
	\node at (30:0.5) {$5$};
	\node at (270:0.5) {$\infty$};
\end{tikzpicture}

%% file: nonLorentz.tex
\begin{tikzpicture}[thick, scale=1.5]
	\coordinate [label=below:$s_1$] (s1) at (-2,0);
	\coordinate [label=below:$s_2$] (s2) at (-1,0);
	\coordinate [label=below:$s_3$] (s3) at (0,0);
	\coordinate [label=below:$s_4$] (s4) at (1,0);
	\coordinate [label=below:$s_5$] (s5) at (2,0);
	
	\foreach \k in {1,...,5} \fill {(s\k)} circle (1pt);
	\draw[dotted] (s1) --node[above]{$-2$} (s2);
	\draw (s2)--(s3)--(s4);
	\draw[dotted] (s4) --node[above]{$-2$} (s5);
\end{tikzpicture}

%% file: LimitDir.bbl
\newcommand{\etalchar}[1]{$^{#1}$}
\providecommand{\bysame}{\leavevmode\hbox to3em{\hrulefill}\thinspace}
\providecommand{\MR}{\relax\ifhmode\unskip\space\fi MR }
\providecommand{\MRhref}[2]{%
  \href{http://www.ams.org/mathscinet-getitem?mr=#1}{#2}
}
\providecommand{\href}[2]{#2}
\begin{thebibliography}{DHR13}

\bibitem[AB08]{abramenko_buildings_2008}
Peter Abramenko and Kenneth~S. Brown, \emph{Buildings}, GTM, vol. 248,
  Springer, New York, 2008.

\bibitem[AVS93]{alekseevskij_geometry_1993}
Dmitri.~V. Alekseevskij, {Ernest}.~B. Vinberg, and Aleksandr.~S. Solodovnikov,
  \emph{Geometry of spaces of constant curvature}, Geometry, {II},
  Encyclopaedia Math. Sci., vol.~29, Springer, Berlin, 1993, pp.~1--138.

\bibitem[Bou68]{bourbaki_elements_1968}
Nicolas Bourbaki, \emph{{G}roupes et alg{\`e}bres de {L}ie. {C}hapitre {4-6}},
  Paris: Hermann, 1968.

\bibitem[Cec08]{cecil_lie_2008}
Thomas~E. Cecil, \emph{Lie sphere geometry}, 2 ed., {Universitext}, Springer,
  New York, 2008.

\bibitem[CG00]{conze_limit_2000}
Jean-Pierre Conze and Yves Guivarc'h, \emph{Limit sets of groups of linear
  transformations}, Sankhy\=a Ser. A \textbf{62} (2000), no.~3, 367--385.

\bibitem[CL14]{chen_lorentzian_2013}
Hao Chen and Jean-Philippe Labb{\'e}, \emph{Lorentzian {C}oxeter groups and
  {B}oyd-{M}axwell {P}ackings}, preprint, {\tt arXiv:1310.8608} (2014), 28 pp.

\bibitem[DHR13]{dyer_imaginary2_2013}
Matthew Dyer, Christophe Hohlweg, and Vivien Ripoll, \emph{Imaginary cones and
  limit roots of infinite {Coxeter} groups}, preprint, {\tt arXiv:1303.6710}
  (March 2013), 63 pp.

\bibitem[Dye11]{dyer_weak_2011}
Matthew Dyer, \emph{On the weak order of {C}oxeter groups}, preprint, {\tt
  arXiv:abs/1108.5557} (August 2011), 37 pp.

\bibitem[Dye13]{dyer_imaginary_2013}
\bysame, \emph{Imaginary cone and reflection subgroups of {C}oxeter groups},
  preprint, {\tt arXiv:1210.5206} (2013), 89 pp.

\bibitem[HLR14]{hohlweg_asymptotical_2014}
Christophe Hohlweg, Jean-{P}hilippe Labb{\'e}, and Vivien Ripoll,
  \emph{Asymptotical behaviour of roots of infinite {C}oxeter groups}, Canad.
  J. Math. \textbf{66} (2014), no.~2, 323--353.

\bibitem[HPR13]{hohlweg_limit_2013}
Christophe Hohlweg, Jean-Philippe Pr{\'e}aux, and Vivien Ripoll, \emph{On the
  limit set of root systems of {C}oxeter groups and {K}leinian groups},
  preprint, {\tt arXiv:1305.0052} (July 2013), 24~pp.

\bibitem[Hum92]{humphreys_reflection_1992}
James~E. Humphreys, \emph{Reflection groups and {C}oxeter groups}, Cambridge
  Studies in Advanced Mathematics. 29 (Cambridge University Press), 1992.

\bibitem[Kra09]{krammer_conjugacy_2009}
Daan Krammer, \emph{The conjugacy problem for {C}oxeter groups}, Group. Geom.
  Dynam. \textbf{3} (2009), no.~1, 71--171.

\bibitem[Lab13]{labbe_polyhedral_2013}
Jean-Philippe Labb\'e, \emph{Polyhedral {C}ombinatorics of {C}oxeter {G}roups},
  Ph.D. thesis, Freie Universit\"at Berlin, July 2013, available at {\tt
  http://www.diss.fu-berlin.de/diss/receive/FUDISS\_thesis\_000000094753},
  pp.~xvi+103.

\bibitem[LP13]{lam_total_2013}
Thomas Lam and Pavlo Pylyavskyy, \emph{Total positivity for loop groups {II}:
  {C}hevalley generators}, Transform. Groups \textbf{18} (2013), no.~1,
  179--231.

\bibitem[LR14]{sagedays}
Jean-Philippe Labb\'e and Vivien Ripoll, \emph{Implementation of geometric
  {C}oxeter systems}, Sage Days 57, The {S}age Development Team, 2014, {\tt
  http://wiki.sagemath.org/days57}.

\bibitem[LT13]{lam_infinite_2013}
Thomas Lam and Anne Thomas, \emph{Infinite reduced words and the {T}its
  boundary of a {C}oxeter group}, preprint, {\tt arXiv:1301.0873} (January
  2013), 26 pp.

\bibitem[Mar07]{marden_outer_2007}
Albert Marden, \emph{Outer circles}, Cambridge University Press, Cambridge,
  2007.

\bibitem[Max82]{maxwell_sphere_1982}
George Maxwell, \emph{Sphere packings and hyperbolic reflection groups}, J.
  Algebra \textbf{79} (1982), no.~1, 78--97.

\bibitem[Pil06]{pilkington_convex_2006}
Annette Pilkington, \emph{Convex geometries on root systems}, Comm. Algebra
  \textbf{34} (2006), no.~9, 3183--3202.

\bibitem[Rat06]{ratcliffe_foundations_2006}
John~G. Ratcliffe, \emph{Foundations of hyperbolic manifolds}, second ed., GTM,
  vol. 149, Springer, New York, 2006.

\bibitem[Rie58]{riesz_clifford_1958}
Marcel Riesz, \emph{Clifford numbers and spinors ({C}hapters {I}--{IV})},
  Lecture Series, No. 38. The Institute for Fluid Dynamics and Applied
  Mathematics, University of Maryland, College Park, Md., 1958.

\bibitem[S{\etalchar{+}}14]{sage}
William~A. Stein et~al., \emph{{S}age {M}athematics {S}oftware ({V}ersion
  6.1.1)}, The {S}age Development Team, 2014, {\tt http://www.sagemath.org}.

\bibitem[SR13]{shafarevich_linear_2013}
Igor~R. Shafarevich and Alexey~O. Remizov, \emph{Linear algebra and geometry},
  Springer, Heidelberg, 2013.

\bibitem[Vin71]{vinberg_discrete_1971}
Ernest~B. Vinberg, \emph{Discrete linear groups that are generated by
  reflections}, Izv. Akad. Nauk {SSSR.} Ser. Mat. \textbf{35} (1971),
  1072--1112.

\end{thebibliography}
